\documentclass{siamart1116}
\pdfoutput=1

\usepackage{amsfonts}
\usepackage{graphicx}
\usepackage{epstopdf}
\usepackage{cite}
\usepackage{algorithmic}
\usepackage{wasysym}
\usepackage{layouts}

\ifpdf
  \DeclareGraphicsExtensions{.eps,.pdf,.png,.jpg}
\else
  \DeclareGraphicsExtensions{.eps}
\fi

\providecommand{\Div}{\operatorname{div}}          

\providecommand{\Dim}{\operatorname{dim}}            
\providecommand{\dim}{\Dim}

\providecommand{\argmin}{\operatorname*{argmin}}  

\newcommand{\Vd}{{\mathbf{d}}}

\newcommand{\Vf}{{\mathbf{f}}}
\newcommand{\Vg}{{\mathbf{g}}}

\newcommand{\Vn}{{\mathbf{n}}}

\newcommand{\Vq}{{\mathbf{q}}}

\newcommand{\Vt}{{\mathbf{t}}}
\newcommand{\Vu}{{\mathbf{u}}}
\newcommand{\Vv}{{\mathbf{v}}}

\newcommand{\Vx}{{\mathbf{x}}}

\providecommand{\Bx}{{\boldsymbol{x}}}

\newcommand{\VD}{{\mathbf{D}}}

\newcommand{\VI}{{\mathbf{I}}}
\newcommand{\VJ}{{\mathbf{J}}}

\providecommand{\Ca}{{\cal A}}
\providecommand{\Cb}{{\cal B}}

\providecommand{\Ci}{{\cal I}}
\providecommand{\Cj}{{\cal J}}

\providecommand{\Co}{{\cal O}}

\providecommand{\Ct}{{\cal T}}

\providecommand{\Cx}{{\cal X}}

\providecommand{\bbN}{\mathbb{N}}

\providecommand{\bbR}{\mathbb{R}}

\newcommand{\DX}{\,\mathrm{d}\Bx}
\newcommand{\dS}{\,\mathrm{d}S}

\newcommand*{\Hm}[2][\defaultdomain]{H^{#2}({#1})}

\newcommand*{\Hone}[1][\defaultdomain]{\Hm[#1]{1}}

\newcommand{\uin}{u_\mathrm{in}}
\newcommand{\uout}{u_\mathrm{out}}

\newcommand{\dOin}{\partial\Omega^\mathrm{in}}
\newcommand{\dOout}{\partial\Omega^\mathrm{out}}
\newcommand{\dOinzero}{\partial\Omega_0^\mathrm{in}}

\newcommand*{\dn}[1]{\frac{\partial {#1}}{\partial \Vn}}

\newcommand{\Uad}{\mathcal{U}_{\mathrm{ad}}}

\newcommand{\Tad}{\mathcal{T}_{\mathrm{ad}}}

\newcommand{\Winfty}{W^{1,\infty}(\bbR^d; \bbR^d)}
\newcommand{\sym}{\mathrm{sym}}

\usepackage{tikz}
\usepackage{standalone}
\usepackage{mathtools}
\DeclareMathOperator{\tr}{{tr}}
\newcommand{\pOmega}{\begin{tikz} \draw[thick] (0,0)--(0.07,0)--(0.02,0.13)--(0.12,0.22)--(0.22,0.13)--(0.17,0)--(0.24,0); \end{tikz}}
\newtheorem{remark}{Remark}
\newtheorem{assumption}{Assumption}

\numberwithin{theorem}{section}

\newcommand{\TheTitle}{Higher-order moving mesh methods for PDE-constrained shape optimization}
\newcommand{\TheShortTitle}{Higher-order moving mesh methods for shape optimization}
\newcommand{\TheAuthors}{A.~Paganini, F.~Wechsung and P. E. Farrell}

\headers{\TheShortTitle}{\TheAuthors}

\title{{\TheTitle}\thanks{Submitted to the editors DATE.
}}

\author{
  Alberto Paganini\footnotemark[2]
  \and
  Florian Wechsung\footnotemark[2]
  \and
  Patrick E. Farrell\thanks{University of Oxford (\email{paganini@maths.ac.ox.uk}, \email{wechsung@maths.ac.ox.uk}, \email{farrellp@maths.ac.ox.uk}).}
}

\usepackage{amsopn}

\ifpdf
\hypersetup{
  pdftitle={\TheTitle},
  pdfauthor={\TheAuthors}
}
\fi

\begin{document}

\maketitle

\begin{abstract}
We present a new approach to discretizing shape optimization problems that generalizes
standard moving mesh methods to higher-order mesh deformations and that is naturally
compatible with higher-order finite element discretizations of PDE-constraints.
This shape optimization method is based on discretized deformation
diffeomorphisms and allows for arbitrarily high resolution of shapes with arbitrary smoothness.
Numerical experiments show that it allows the solution of PDE-constrained
shape optimization problems to high accuracy.
\end{abstract}



\section{Introduction}\label{sec:intro}
Shape optimization problems are optimization problems {where} the control to
be optimized is the shape of a domain. Their basic formulation generally reads
\begin{equation}\label{eq:shapeoptproblemnoconstraint}
\text{find }\Omega^*\in\argmin_{\Omega \in \Uad} \Cj(\Omega)\,,
\end{equation}
where $\Uad$ denotes a collection of \emph{admissible shapes} and
$\Cj: \Uad \to\bbR$ represents a \emph{shape functional}.
In many applications, the shape functional depends not only on the shape of a domain
$\Omega\subset{\bbR^d}$,
but also on the solution $u$ of a boundary value problem (BVP) posed on $\Omega$,
in which case \cref{eq:shapeoptproblemnoconstraint} becomes
\begin{subequations}\label{eq:shapeoptproblem}
\begin{equation}\label{eq:shapefunctional}
\text{find }\Omega^*\in\argmin_{\Omega \in \Uad} \Cj(\Omega, u_\Omega)
\quad \text{subject to}
\end{equation}
\begin{equation}\label{eq:stateconstr}
u_\Omega\in V(\Omega)\,, \quad
a_\Omega(u_\Omega, v) = f_\Omega(v) \quad \text{for all } v\in W(\Omega)\,,
\end{equation}
\end{subequations}
where \cref{eq:stateconstr} represents the variational formulation of a BVP that acts as a 
PDE-constraint.

These  problems are said to be \emph{PDE-constrained} and are notoriously difficult
to solve because the dependence of $\Cj$ on the domain is nonconvex. Additionally, the function
$u_\Omega$ cannot be computed analytically. Even approximating it with a numerical method is
challenging because the computational domain of the PDE-constraint is the unknown variable to be
solved for in the shape optimization problem.

The literature abounds with numerical methods for BVPs. Here, we consider approximation
by means of finite elements, which has become the most popular choice for PDE-constrained 
shape optimization
due to its flexibility for engineering applications. Nevertheless, it is worth mentioning that alternatives 
based on other discretizations have also been considered
\cite{EpHa12, BaCiOfStZa15, AnBiVe15,ScIlScGa13}.

Most commonly, PDE-constrained shape optimization problems are formulated
with the aim of further improving the performance of an initial design $\Omega^0$.
The standard procedure to pursue this goal  is to iteratively update
some parametrization of $\Omega^0$ to decrease the value of $\Cj$.
Obviously, the choice of this parametrization has an enormous influence on
the design of the related shape optimization algorithm and on the search space $\Uad$ itself.

In this work, we parametrize shapes by applying deformation
diffeomorphisms to the  initial guess $\Omega^0$.
In this framework, solving a shape optimization problems translates into
constructing an optimal diffeomorphism.
To construct this diffeomorphism with numerical methods, we introduce a
discretization of deformation vector fields.
This approach can be interpreted as a higher-order generalization of
standard moving mesh methods and can be combined with isoparametric finite elements
to obtain a higher-order discretization of the PDE-constraint.

There are several advantages to using
higher-degree and smoother transformations.
First, higher-degree parametrization of domains allows for the
consideration of more general shapes (beyond polytopes).
Secondly, the efficiency of a higher-order discretization of a BVP hinges on
the regularity of its solution, which depends on the regularity of the
computational domain, among other factors.
Finally, a smoother 
discretization of deformation vector fields allows the computation of
more accurate Riesz representatives of shape derivatives \cite{Pa17, HiPaSa15}, and thus,
more accurate descent directions for shape optimization algorithms.

Our approach is generic and allows for the discretization of domain transformations based on
B-splines, Lagrangian finite elements, or harmonic functions, among others.
This discretization can comprise arbitrarily many basis functions and thus allow for arbitrarily
high resolution of shapes with arbitrary smoothness.
Moreover, because our approach decouples the discretization of the state and the
control, it is straightforward to implement, and requires typically no modification of existing finite element software. 
This is a significant advantage for practical applications.

The rest of this article is organized as follows.
In \cref{sec:formulation}, we describe how we model the search space $\Uad$ 
with deformation diffeomorphisms and discuss the advantages and disadvantages of this choice.
In \cref{sec:optimization}, we give a brief introduction to \emph{shape calculus} and
explain how to compute steepest-descent updates for deformation diffeomorphisms
using shape derivatives of shape functionals.
In \cref{sec:PDEdJ}, we emphasize that having a PDE constraint necessitates the
solution of a BVP and its adjoint at each step of the optimization, and comment on the error introduced by their finite element discretizations.
In  \cref{sec:isofem}, we give an introduction to isoparametric finite elements and explain
why it is natural to employ this kind of discretization to approximate the state variable $u_\Omega$ 
when the domain $\Omega$ is modified by a diffeomorphism.
In \cref{sec:implementation}, we examine implementation aspects of the algorithm
resulting from \cref{sec:optimization} and \cref{sec:isofem}.
In particular, we give detailed remarks for an efficient implementation of a decoupled discretization
of the state and control variables.
Finally, in \cref{sec:numexp}, we perform numerical experiments.
On the one hand, we consider a well-posed test case and investigate the impact of the discretization
of the state and of the control variables on the performance of higher-order moving mesh methods,
showing that these methods can be employed to solve PDE-constrained shape optimization problems
to high accuracy.
On the other hand, we consider more challenging PDE-constrained shape optimization problems and
show that the proposed shape optimization method is not restricted to a specific problem.
\begin{remark}
Shape optimization problems with a distinction between computational domain
and control variable also exist.
For instance, this is the case for PDE-constrained optimal control problems
where the control is a piecewise constant coefficient in the PDE-constraint \cite{Pa16, LaSt16},
in which case the control is the shape of the contour levels of the piecewise constant coefficient.
The approach suggested in this work covers already this more general type of shape optimization problem.
However, to simplify the discussion and reduce the amount of technicalities,
we consider only problems of the form
\cref{eq:shapeoptproblem}.
\end{remark}
\section{Parametrization of shapes via diffeomorphisms}\label{sec:formulation}
Among the many possibilities for defining $\Uad$,
we choose to construct it by collecting all domains that can be obtained by applying
(sufficiently regular) geometric transformations\footnote{A geometric transformation
is a bijection from $\bbR^d$ onto itself.} to an initial domain $\Omega^0$, that is,
\begin{equation}\label{eq:Uad}
\Uad \coloneqq \{T(\Omega^0):\,T\in\Tad\}\,,
\end{equation}
where $\Tad$ is (a subgroup of) the group of $W^{1,\infty}$-diffeomorphisms.
We recall that $\Winfty$ is the Sobolev space of locally integrable vector fields with
essentially bounded weak derivatives. We impose this regularity requirement on $\Tad$
to guarantee that  the state constraint \cref{eq:stateconstr} is well-defined for every domain
in $\Uad$ (assuming that it is well-defined on the initial domain $\Omega^0$).
Note that it may be necessary to strengthen the regularity requirements on $\Tad$
if the PDE-constraint is a BVP of higher order such as, for example, the biharmonic equation
\cite{BaMa14}.

While there are many alternatives (for instance level sets \cite{AlJoTo02}
or phase fields \cite{GaHeHiKaLa16}), we prefer to describe $\Uad$
as in \cref{eq:Uad} because it incorporates an explicit description of the boundaries
of the domains contained within it.
In fact, describing shapes via diffeomorphisms is a standard approach in shape optimization;
cf. \cite[Chap. 3]{DeZo11}.
From a theoretical point of view, it is possible to impose a metric on \cref{eq:Uad} and to
investigate the existence of optimal solutions within this framework \cite{SiMu76}.
Recently, the convergence of Newton's method in this framework has also been investigated \cite{St16}.

\begin{remark}
The representation of domains via transformations in \cref{eq:Uad} is not unique, 
and it is generally possible to find two transformations $T_1\in \Tad$ and $T_2\in \Tad$
such that $T_1 \neq T_2$ and $T_1(\Omega^0) = T_2(\Omega^0)$. For instance, this is the
case if $\Omega^0$
is a ball and $T_2=T_1\circ T_R$, where $T_R$ is a rotation around the center of $\Omega^0$.
To obtain a one-to-one correspondence between shapes and transformations,
one can introduce equivalence classes, but this is not particularly relevant for this work.
\end{remark}

To shorten the notation, we introduce the reduced functional
\begin{equation}\label{eq:redfct}
j:\Uad \to \bbR\,,\quad \Omega \mapsto \Cj(\Omega,u_{\Omega})\,,
\end{equation}
which is well-defined under the following assumption on the PDE-constraint \cref{eq:stateconstr}.\begin{assumption}\label{ass:welldefBVP}
Henceforth, we assume that the BVP \cref{eq:stateconstr} that acts as PDE-constraint is
well-defined in the sense of Hadamard: for every $\Omega\in\Uad$
the BVP \cref{eq:stateconstr} has a unique solution $u_\Omega$ that depends continuously on the 
BVP data.
\end{assumption}

In \cref{sec:intro}, we mentioned that shape optimization problems are solved
updating iteratively some parametrization of $\Omega^0$, that is,
constructing a sequence of domains $\{\Omega^{(k)}\}_{k\in\bbN}$ so that
$\{j(\Omega^{(k)})\}_{k\in\bbN}$ decreases monotonically.
For simplicity, we relax the terminology and call such a sequence \emph{minimizing},
although the equality
\begin{equation}\label{eq:minseq}
\lim_{k\to \infty}j(\Omega^{(k)}) = \inf_{\Omega\in\Uad} j(\Omega)\,.
\end{equation}
may not be satisfied.

When the search space $\Uad$ is constructed as in \cref{eq:Uad},
computing $\{\Omega^{(k)}\}_{k\in\bbN}$ translates into creating
a sequence of diffeomorphisms $\{T^{(k)}\}_{k\in\bbN}$. Generally, 
the sequence $\{T^{(k)}\}_{k\in\bbN}$ is constructed according to the following procedure:
\begin{itemize}
\label{test}\item[(a)] given the current iterate $T^{(k)}$, derive a tentative iterate $\tilde{T}$,
\item[(b)] if $\tilde{T}$ satisfies certain quality criteria, set $T^{(k+1)} = \tilde{T}$ and move to the next step; 
otherwise compute another $\tilde{T}$.
\end{itemize}
In the next section, we discuss the computation of $\tilde{T}$ with shape derivatives.
For simplicity, we first assume that the state variable $u_\Omega$ is known analytically
and restrict our considerations to the reduced functional $j$ defined in \cref{eq:redfct}.
The role of the PDE-constraint and the discretization of the state
variable is discussed in \cref{sec:PDEdJ}.

\section{Iterative construction of diffeomorphisms}\label{sec:optimization}
Shape calculus offers an elegant approach for constructing
a minimizing sequence of domains $\{\Omega^{(k)}\}_{k\in\bbN}$.
The key tool is the derivative of the shape functional $\Cj$ with respect
to shape perturbations. To give a more precise
description, let us first introduce the operator
\begin{equation}\label{eq:Jtilde}
J: \Tad \to\bbR\,, \quad T \mapsto j(T(\Omega^0))\,.
\end{equation}
Since $\Tad\subset \Winfty$, which is a Banach space with respect to the norm \cite[Sect. 5.2.2]{Ev10}
\begin{equation}
\Vert T \Vert_{\Winfty} \coloneqq
\sum_{\vert \alpha \vert \leq 1} \mathrm{ess}\sup \Vert \VD^{\alpha} T \Vert\,,
\end{equation}
we can formally define the directional derivative of $J$ at $T\in \Tad$ in the direction
$\Ct\in\Winfty$ through the limit
\begin{equation}\label{eq:dirder}
dJ(T; \Ct) \coloneqq \lim_{s\to 0^+} \frac{J((\Ci+s \Ct)\circ T) - J(T)}{s}
= \lim_{s\to 0^+} \frac{J(T+s \Ct\circ T) - J(T)}{s}\,.
\end{equation}
\begin{remark}
Note that $\Ci+s \Ct$ is a $W^{1,\infty}$-diffeomorphism for sufficiently small $s$ 
\cite[Lemma 6.13]{Al07}.
\end{remark}

A shape functional $\Cj$ is said to be \emph{shape differentiable} (in $T(\Omega^0)$)
if the corresponding functional \cref{eq:Jtilde} is Fr\'{e}chet differentiable (in $T$), that is, if
\cref{eq:dirder} defines a linear continuous operator on $\Winfty$
such that
\begin{equation}
\vert J(T+s \Ct\circ T)\ - J(T)- dJ(T; s\Ct)\vert = o(s)
\quad \text{for all } \Ct\in \Winfty\,.
\end{equation}

\begin{remark}
Generally, \cref{ass:welldefBVP} is not sufficient to guarantee that
$\Cj$ is shape differentiable. In particular, it is necessary to ensure that
the solution operator $\Omega \mapsto u_\Omega$ is continuously differentiable;
cf. \cite[Sect. 1.6]{HiPiUlUl09}.
\end{remark}

The Fr\'{e}chet derivative $dJ$ can be used to construct a sequence of 
diffeomorphisms $\{T^{(k)}\}_{k\in\bbN}$ to solve \cref{eq:shapeoptproblem}
in a steepest descent fashion.
More specifically, the entries of this sequence take the form
\begin{equation}\label{eq:Tsequence}
T^{(0)}(\Vx) = \Vx\quad \text{and} \quad T^{(k+1)}(\Vx) = (\Ci + \mathrm{d}T^{(k)})\circ(T^{(k)}(\Vx)),
\end{equation}
where the update $\mathrm{d}T^{(k)} :\bbR^d \to \bbR^d$ is computed with the help of $dJ$.
For instance, we could define \cite[Page 103]{HiPiUlUl09}
\begin{equation}\label{eq:dTinf}
dT^{(k)} \in
\argmin_{\substack{\Ct \in \Winfty\,,\\ \Vert \Ct\Vert_{W^{1,\infty}}=1}} dJ(T^{(k)}; \Ct)\,.
\end{equation}
Unfortunately, such a descent direction may not exist without making further assumptions
on $dJ$ because $\Winfty$ is not reflexive; cf. \cite{Ja63}. However, in \cite{LaSt16} it has been shown
that in many instances (and under suitable assumptions), the operator $dJ$ takes the form
\begin{equation}\label{eq:dJtensor}
dJ(T^{(k)}; \Ct) = \int_{T^{(k)}(\Omega_0)} \sum_{i,j=1}^d s_1^{i,j} \VD\Ct^{i,j} +  \sum_{\ell=1}^d s_0^\ell \Ct^\ell\DX\,,
\end{equation}
where $s_1^{i,j}$, $i,j=1,\dots, d$ and  $s_0^{\ell}$, $\ell=1,\dots, d$ are (instance dependent)
$L^{1}(\bbR^d)$-functions.
The following proposition\footnote{We provide a full proof of this proposition because, to the best of
our knowledge, this result is new.} states that, in this case, \eqref{eq:dTinf} can be used to define
steepest descent directions. 
\begin{proposition}\label{prop:dtexists}
Let $dJ$ be as in \eqref{eq:dJtensor}. Then, there exists a descent direction $dT^{(k)}$ as defined in \cref{eq:dTinf}
\end{proposition}
\begin{proof}
First of all, we recall that $L^{\infty}(D)$ is isometrically isomorphic to the dual $X^*$ of 
$X = L^{1}(D)$ (for any open domain $D\subset\bbR^m$ in any fixed dimension $m$).
We denote by $\phi_D : L^{\infty}(D) \to X^*$ this isomorphism.
The duality pairing $\langle\cdot,\cdot\rangle_{X^*\times X}$ can be characterized by
\begin{equation}\label{eq:dualitypairing}
\langle f,g\rangle_{X^*\times X} = \int_D \phi_D^{-1}(f) g \DX\,.
\end{equation}
Clearly, similar pairings exist for Cartesian products of $L^{\infty}(D)$.
Finally, note that $L^{1}(D)$ is separable. Therefore,  by the Banach-Alaoglu theorem,
any bounded sequence in $L^{\infty}(D)$ has a subsequence $\{\Vx_n\}$ that converges
weakly-* to an $\Vx\in L^{\infty}(D)$, that is,
\begin{equation}
\lim_{n\to\infty} \int_D \Vx_n g \DX = \int_D \Vx g \DX\quad \text{for every } g\in L^{1}(D)\,.
\end{equation}
Using these results, we show that a steepest descent direction exists.

Let $\Ct_n$ be a minimizing sequence of \cref{eq:dTinf}. By definition,
$\Ct_n$ is bounded in $\Winfty$, and hence in $L^{\infty}(\bbR^d,\bbR^d)$, too.
Therefore, there exists a subsequence $\Ct_{n_k}$
that converges weakly-* to a $T\in L^{\infty}(\bbR^d,\bbR^d)$.
Since $\Ct_{n_k}$ is bounded in $\Winfty$, there is a subsequence $\Ct_{n_{k_\ell}}$
such that $\VD\Ct_{n_{k_\ell}}$ converges weakly-* in $L^{\infty}(\bbR^d,\bbR^{d,d})$.
By the definition of weak derivative, it is easy to see that the weakly-* limit of
$\VD\Ct_{n_{k_\ell}}$ is $\VD T$.
This shows that $T\in W^{1,\infty}(\bbR^d,\bbR^d)$.

Since \cref{eq:dJtensor} is a sum of duality pairings as in \cref{eq:dualitypairing}
(with $D=\bbR^d$ and $g=\chi_{T^{(k)}(\Omega_0)} s_{1}^{i,j}$ or $g=\chi_{T^{(k)}(\Omega_0)} s_{0}^{\ell}$,
where $\chi_{T^{(k)}(\Omega_0)}$ is the characteristic function associated to $T^{(k)}(\Omega_0)$),
$dJ$ is weakly-* continuous.
Therefore, $T$ is a minimizer, because it is the weak-* limit
of $\Ct_{n_{k_\ell}}$, which is a subsequence of a minimizing sequence.

Finally, to show that $\Vert T \Vert_{\Winfty}=1$, we recall that the norm of a Banach space is weak-*
lower semi-continuous. Therefore,
\begin{equation}
\Vert T \Vert_{\Winfty} \leq \liminf_{k\to\infty} \Vert \Ct_{n_k} \Vert_{\Winfty} \leq 1\,,
\end{equation}
and since $dJ(T^{(k)};\cdot)$ is linear, $\Vert T \Vert_{\Winfty} = 1$.

To conclude, note that $dJ(T^{(k)};T)>-\infty$ because $dJ(T^{(k)};\cdot)$ is continuous.
\end{proof}

Although possibly well-defined, it is challenging to compute such a descent direction $dT^{(k)}$
(because $\Winfty$ is infinite dimensional and neither reflexive nor separable).
One possible remedy is to introduce a Hilbert subspace 
$(\Cx, (\cdot,\cdot)_\Cx)$, $\Cx \subset \Winfty$, and to compute
\begin{equation}\label{eq:dTX}
dT^{(k)}_\Cx \coloneqq 
\argmin_{\substack{\Ct \in \Cx\,,\\ \Vert \Ct\Vert_{\Cx}=1}} dJ(T^{(k)}; \Ct)\,,
\end{equation}
that is, the gradient of $dJ$ with respect to $(\cdot,\cdot)_\Cx$.
Up to a scaling factor, the solution of \cref{eq:dTX} can be computed by solving
the variational problem: find $dT^{(k)}$ such that
\begin{equation}\label{eq:dTXBVP}
(dT^{(k)}, \Ct)_\Cx =  -dJ(T^{(k)}; \Ct)\quad \text{for all }\Ct \in \Cx\,,
\end{equation}
which is well-posed by the Riesz representation theorem.
However, the condition $\Cx \subset \Winfty$ is restrictive (for a Hilbert space).
For instance, the general Sobolev inequalities guarantee that
the Sobolev space $H^k(\bbR^d;\bbR^d)$ is contained in $\Winfty$
only for $k\geq d/2+1$ \cite[Sect. 5.6.3]{Ev10}.

A more popular approach is to introduce a
finite dimensional subspace $Q_N\subset \Cx \cap \Winfty$
and to compute the solution of
\begin{equation}\label{eq:dTQn}
(dT^{(k)}_N, \Ct_N)_\Cx =  -dJ(T^{(k)}; \Ct_N)\quad \text{for all }\Ct_N \in Q_N\,.
\end{equation}
In this case, the requirement  $\Cx \subset \Winfty$
can be dropped as long as the dimension $N\coloneqq\dim(Q_N)$ of $Q_N$ is finite.
However, note that if $\{Q_N\}_{N\in\bbN}$ is a family of nested finite dimensional spaces
such that $\overline{\cup_{N\in\bbN} Q_N}^{\Cx}=\Cx$,
the sequence $dT^{(k)}_N$ can be interpreted as the Ritz-Galerkin approximation of $dT^{(k)}$.
Therefore, as $N\to \infty$, the sequence $dT^{(k)}_N$ may converge to an element of
$\Cx\setminus \Winfty$, which does not qualify as an admissible update.

The trial space $Q_N$ can be constructed with
linear Lagrangian finite elements defined on (a mesh of) a hold-all domain $D\supset\Omega^0$.
The resulting algorithm is equivalent to standard moving mesh methods \cite{Pa17}. 
Alternatively, one can employ tensorized B-splines \cite{HiPa15}.
Lagrangian finite elements have the advantage of inclusion in standard finite element software,
whereas B-splines offer higher regularity, which is often desirable (as we will argue in \cref{sec:implementation}).
For instance, univariate B-splines of degree $d$ are in $W^{d, \infty}(\bbR)$ \cite{Ho03},
whereas Lagrangian finite elements are not even $C^1$.
As for the Hilbert space $\Cx$, one usually opts for $\Cx = H^1(D)$
or, equivalently, for $H^{1/2}(\partial\Omega^{(k)})$ combined with an elliptic extension operator
onto $D$ \cite{ScSiWe16}. This choice can be motivated by considerations of the shape
Hessian \cite{ScSiWe16,EpHa12}.

To the best of our knowledge, it has not been settled yet which definition of 
steepest direction among \cref{eq:dTinf,eq:dTX,eq:dTQn} is best suited to formulate
a numerical shape optimization algorithm. Since the focus of this work is more on the discretization of
shape optimization problems than on actual optimization algorithms, we
postpone investigations of this topic to future research. In our numerical experiments
in \cref{sec:numexp}, we will employ \cref{eq:dTQn}, which is the computationally most
tractable definition. However, note that computing steepest directions according to
\cref{eq:dTinf} or \cref{eq:dTX} would also inevitably involve some discretization, because
$\Winfty$ (and generally $\Cx$) are infinite dimensional.

We conclude this section with the Hadamard homeomorphism theorem
\cite[Thm 1.2]{Ka94}, which gives explicit criteria to verify that
the entries of the sequence $\left\{ÊT^{(k)}\right\}_{k\in\bbN}$
defined in \cref{eq:Tsequence} are admissible transformations.
\begin{theorem}\label{thm:Hadamard}
Let $X$ and $Y$ be finite dimensional Euclidean spaces, and let 
$T:X\to Y$ be a $C^1$-mapping that satisfies the following conditions:
\begin{enumerate}
\item $\det(\VD T)(x) \neq 0$ for all $x\in X$.
\item $\Vert T(x) \Vert \to \infty$ as $\Vert x \Vert \to \infty$.
\end{enumerate}
Then $T$ is a $C^1$-diffeomorphism from $X$ to $Y$.
\end{theorem}
A counterpart of \cref{thm:Hadamard} for $W^{1,\infty}$-transformations
can be found in \cite{GaRa15}.

For the sequence \cref{eq:Tsequence}, note that the second hypothesis of \cref{thm:Hadamard}
is automatically satisfied if the hold-all domain $D$ is bounded, because the update
$\mathrm{d}T^{(k)}$ has compact support, and
$T^{(k+1)}(\Vx) = \Vx$ for every $\Vx\in \bbR^d\setminus\overline{D}$.

\section{Shape derivatives of PDE-constrained functionals}\label{sec:PDEdJ}

To simplify the exposition,
in the previous section we treated the dependence of $J$ on $u$ implicitly;
this dependence was hidden in the reduced functional $j$.
We now examine the consequences of this dependence, as it introduces additional difficulties.
Indeed,
it is generally the case that to evaluate the Fr\'{e}chet derivative $dJ$
it is necessary to solve (at least) one BVP.
To illustrate this fact, we
consider the following example:
\begin{subequations}\label{eq:examplePDEJ}
\begin{equation}
\Cj(\Omega, u_\Omega) = \frac{1}{2} \int_\Omega u_\Omega^2\DX\,, \quad\text{subject to}
\end{equation}
\begin{equation}
u_\Omega\in\Hone\,,\quad\int_\Omega \nabla u_\Omega\cdot\nabla v+ u_\Omega v\DX 
= \int_\Omega v \DX \quad \text{for all }Êv\in \Hone\,.
\end{equation}
\end{subequations}
Its shape derivative reads \cite[Eq. 2.12]{Pa16}
\begin{align}\label{eq:examplePDEdJ}
dJ(T;\Ct) = \int_{T(\Omega^0)} \big(&\nabla u_{T(\Omega^0)} \cdot (\VD\Ct+\VD\Ct^\top)\nabla p\\
\nonumber
&+(p+u_{T(\Omega^0)}^2 -\nabla u_{T(\Omega^0)}\cdot\nabla p - u_{T(\Omega^0)}p)\Div\Ct \big) \DX\,,
\end{align}
where $p\in H^1(T(\Omega^0))$ is the solution of the adjoint BVP
\begin{equation}\label{eq:exampleAdjBVP}
\int_{T(\Omega^0)} \nabla p\cdot\nabla v+ p v\DX 
= \int_{T(\Omega^0)} u_{T(\Omega^0)}v \DX \quad \text{for all }Êv\in H^1(T(\Omega^0))\,.
\end{equation}
Formula \cref{eq:examplePDEdJ} clearly shows that it is necessary to compute the functions $u_{T(\Omega^0)}$ and $p$ to evaluate $dJ$.
The adjoint BVP \cref{eq:exampleAdjBVP} is introduced to derive a formula of $dJ$
that does not contain the shape derivative of $u_\Omega$. This is a well-known strategy
in PDE-constrained optimization \cite[Sect. 1.6]{HiPiUlUl09}.

In general, deriving explicit formulae for Fr\'{e}chet derivatives of PDE-constrained functionals is a
delicate and error prone task. However, in many instances one can
introduce a Lagrangian functional that allows the automation of the differentiation process 
and gives the correct adjoint equations \cite[Sect. 1.6.4]{HiPiUlUl09}.
The level of automation is such that numerical software is capable of differentiating several
PDE-constrained functionals \cite{FaHaFuRo13}.
Clearly, Lagrangians are useful also for the special case of PDE-constrained shape
functionals \cite[Chap. 10]{DeZo11}, and dedicated numerical software for shape differentiation
has recently become available \cite{Sc16}.

\begin{remark}
The Hadamard-Zol\'{e}sio structure theorem \cite[Chap. 9, Thm 3.6]{DeZo11}
states that, under certain regularity assumptions on $\Omega$,
the Fr\'{e}chet derivative $dJ(\Omega;\Ct)$ depends only on perturbations 
$\Ct(\partial\Omega)$ of the domain boundary. 
As a consequence, the derivative of most shape functionals can be formulated as an integral
both in the volume $\Omega$ and on the boundary $\partial\Omega$,
and these formulations are equivalent. For instance, the boundary formulation that corresponds
to \cref{eq:examplePDEdJ} reads \cite[Eq. 2.13]{Pa16}
\begin{equation}
\int_{T(\partial\Omega^0)} \Ct \cdot \Vn \left( u_{T(\Omega^0)}^2
-\nabla u_{T(\Omega^0)}\cdot\nabla p -u_{T(\Omega^0)}p+p\right) \dS\,.
\end{equation}
When the state and the adjoint variables are replaced by numerical approximations,
these two formulae define two different approximations of $dJ$.
In the framework of finite elements, it has been shown that volume based formulations
usually offer higher accuracy compared to their boundary based counterparts
\cite{HiPaSa15, Pa15}.
Additionally, the combination of volume based formulae with piecewise linear
finite element discretization of the control variable results in shape optimization algorithms
for which the paradigms \emph{optimize-then-discretize} and \emph{discretize-then-optimize}
commute \cite{HiPa15, Be10}. This does not hold in general for boundary based formulae
because piecewise linear finite elements do not fulfill the necessary regularity requirements,
and the equivalence of boundary and volume based formulae is not guaranteed \cite{Be10}.
\end{remark}
\section{Isoparametric Lagrangian finite elements}\label{sec:isofem}

To evaluate the shape functional $\Cj(\Omega, u_\Omega)$,
it is necessary to approximate the function $u_\Omega$, which is the solution of the PDE-constraint \cref{eq:stateconstr}. For the Fr\'{e}chet derivative $dJ$, it may be necessary to
also approximate the solution $p$ of an adjoint BVP.
In this work, we consider the discretization of \cref{eq:stateconstr} and the adjoint BVP
by means of finite elements.
Finite element spaces are defined on meshes of the computational domain.
As shape optimization algorithms modify the computational domain,
a new mesh is required at each iteration. This new mesh can either be constructed
\emph{de novo} or by modifying a previously existing mesh. On the one hand, remeshing should
be avoided
because it is  computationally expensive and may introduce undesirable 
noise in the optimization algorithm. On the other hand, updating the mesh is a delicate process
and may return a mesh with poor quality (which in turn introduces noise in the optimization as well).

Isoparametric finite elements offer an interesting perspective on the process of mesh updating
that fits well with our encoding of changes in the domain via geometric transformations.
In particular, with isoparametric finite elements it is possible to mimic the modification of the 
computational domain without tampering directly with the finite element mesh. Additionally, 
isoparametric finite element theory
provides insight into the extent to which remeshing can be avoided. Next, we provide a concise recapitulation of isoparametric finite element theory. 
For simplicity, we assume that the PDE-constraint is a
linear $V$-elliptic second-order BVP. However, we believe that most of the considerations
readily cover more general BVPs.
For a more thorough introduction to isoparametric finite elements, we refer to \cite[Sect. 4.3]{Ci02}.

The Ritz-Galerkin discretization of \cref{eq:stateconstr} reads
\begin{equation}\label{eq:stateconstrFE}
\text{find}\quad u_h\in V_h(\Omega)\,, \quad
a_\Omega(u_h, v_h) = f_\Omega(v_h) \quad \text{for all } v_h\in V_h(\Omega)\,,
\end{equation}
where $V_h(\Omega)$ is a finite dimensional subspace of $V(\Omega)=W(\Omega)$.
Henceforth, we restrict ourselves to Lagrangian finite element approximations
on simplicial meshes.

Let us assume for the moment that $\Omega$ is a polytope. 
The most common construction of finite element spaces begins with a triangulation $\Delta_h(\Omega)$
of $\Omega$.
This triangulation is used to introduce global basis functions that span the finite element space.
The finite element space is called Lagrangian if the degrees of freedom of its global basis
functions are point evaluations \cite[Page 36]{Ci02}, and it is called of degree $p$
if the local basis functions, that is, the restriction of global basis functions to elements $K$
of the triangulation, are polynomials of degree $p$.

It is well known that Lagrangian finite elements on simplicial meshes are \emph{affine equivalent}. 
Affine equivalence means that we can define a reference element $\hat{K}$ and a set of reference
local basis functions $\{\hat{b}_i\}_{i\leq M}$ on $\hat{K}$, and construct a family of affine diffeomorphisms
$\{ G_K: \hat{K} \to K\}_{K\in\Delta_h(\Omega)}$ such that
the local basis functions $\{b_i^K\}_{i\leq M}$ on $K$ satisfy $b_i^K(\Vx) = \hat{b}_i(G_K^{-1}(\Vx))$.
Note that both $\{b_i^K\}_{i\leq M}$ and $\{\hat{b}_i\}_{i\leq M}$ are polynomials, because
the pullback induced by a bijective affine transformation is an automorphism.

Issues with this construction arise if $\Omega$ has curved boundaries. In this case,
we introduce first an affine equivalent finite element space $V_h(\pOmega)$ built on
the triangulation $\Delta_h(\pOmega)$ of a polytope $\pOmega$ that approximates
$\Omega$.
Then, we construct a vector field $F\in (V_h(\pOmega))^d$ such that 
$F(\partial\pOmega)\approx \partial \Omega$ and generate a (curved) triangulation
$\Delta_h(\Omega)$ by deforming the elements of $\Delta_h(\pOmega)$ according to $F$.
Finally, we define the finite element space $V_h(\Omega)$ on $\Delta_h(\Omega)$
by choosing $b_i^K(\Vx) = \hat{b}_i(G_K^{-1}(F^{-1}(\Vx)))$ as local basis functions.
This construction leads to so-called isoparametric finite elements. Again, this space is called
Lagrangian if the reference local basis functions $\{\hat{b}_i\}_{i\leq M}$ are polynomials.
However, note that the local basis functions $\{b_i\}_{i\leq M}$ 
of isoparametric Lagrangian finite elements may not be polynomials.

Isoparametric Lagrangian finite elements on curved domains
are proved to retain the approximation properties of Lagrangian finite elements
on polytopes under the following additional assumptions \cite[Thm 4.3.4]{Ci02}:
\begin{enumerate}
\item the triangulation $\Delta_h(\pOmega)$ is regular \cite[Page 124]{Ci02},
\item the mesh width $h$ is sufficiently small,
\item for every quadrature point $\Vx_q\in\hat{K}$, and for every element $K\in\Delta_h(\pOmega)$ 
\begin{equation}\label{eq:distortion} 
\Vert F(G_K(\Vx_q)) - G_K(\Vx_q) \Vert = \Co(h^p)\,,
\end{equation}
and $F(G_K(\Vx_q)) \in \partial\Omega$ whenever $G_K(\Vx_q)\in \partial\pOmega$.
\end{enumerate}
Equation \cref{eq:distortion} is sufficient to guarantee that the map $F\circ G_K$
is a diffeomorphism, and to provide algebraic estimates of the form
\begin{equation}\label{eq:Fdecay}
\Vert \VD^{\alpha} (F\circ G_K) \Vert  = \Co(h^\alpha)\,,
\end{equation}
which are necessary to derive the desired approximation estimates.

This knowledge of isoparametric finite elements is sufficient to tackle our initial problem:
solve \cref{eq:stateconstrFE} on $\Omega^{(k)}$
(where $\Omega^{(k)}\coloneqq T^{(k)}(\Omega^0)$).

In the first iteration, we construct $V_h(\Omega^0)$ in the isoparametric fashion described above.
First, we generate a triangulation of a suitable polytope
$\pOmega^0$ that approximates $\Omega^0$.
Then, we define the finite element space $V_h(\pOmega^0)$
and generate a transformation $F^{(0)}\in(V_h(\pOmega^0))^d$ that maps $\pOmega^0$ 
onto $\Omega^0$. Finally, we construct $V_h({\Omega}^0)$ by combining 
reference local basis functions with the diffeomorphism $F^{(0)}$.

In the next iteration, we construct $V_h(\Omega^{(1)})$ 
in the same way, but replacing
the diffeomorphism $F^{(0)}$ with the interpolant 
$$(V_h(\pOmega^0))^d\ni F^{(1)} \coloneqq \Ci_h (T^{(1)}\circ F^{(0)})\,,$$
where $\Ci_h$ denotes the interpolation operator onto $(V_h(\pOmega^0))^d$.
Since
$$T^{(1)}(\Vx) = \Vx + \mathrm{d}T^{(0)}(\Vx)\,,$$
the map $F^{(1)}$ can be written as
$$F^{(1)} = F^{(0)} + \Ci_h(\mathrm{d}T^{(0)}\circ F^{(0)}) \,.$$

Repeating this procedure at every iteration results in the isoparametric space
$V_h(\Omega^{(k)})$ being constructed with the map
\begin{align}
\nonumber F^{(k)} &=  \Ci_h(T^{(k)}\circ F^{(0)}) \,,\\
\nonumber&= \Ci_h(T^{(k-1)}\circ F^{(0)}) + \Ci_h(\mathrm{d}T^{(k-1)}\circ T^{(k-1)}\circ F^{(0)})\,,\\
\label{eq:Fk}&= F^{(k-1)} + \Ci_h(\mathrm{d}T^{(k-1)}\circ F^{(k-1)})\,,
\end{align}
where the second equality follows from $F^{(k-1)} = \Ci_h (T^{(k-1)}\circ F^{(0)})$.

In general, the map $F^{(k)}$ may not fulfill the condition \cref{eq:distortion}.
However, by $W^{1,\infty}$-error estimates of $\Ci_h$ \cite[Thm 4.3.4]{Ci02}, 
it holds that
$$\det(\VD F^{(k)})(x) \to \det(\VD (T^{(k)}\circ F^{(0)}))(x) \text{ as } h\to 0\,.$$
This, in light of \cref{thm:Hadamard}, guarantees that $F^{(k)}$
is indeed a diffeomorphism if $h$ is small enough (because $T^{(k)}\circ F^{(0)}$ is 
a diffeomorphism as well, and therefore $\det(\VD (T^{(k)}\circ F^{(0)}))(x)\neq 0$).
Additionally, note that the element transformation $G_K:\hat{K}\to K$ is affine,
and thus,
$$\VD^\alpha (F^{(k)}\circ G_K) = (\VD^\alpha (F^{(k)})\circ G_K)(\VD G_K)^{\alpha}\,.$$
Therefore,
\begin{equation}\label{eq:Fkdecay}
\Vert \VD^{\alpha} (F^{(k)}\circ G_K) \Vert 
\leq \Vert (\VD^\alpha (F^{(k)})\circ G_K)\Vert \Vert (\VD G_K)^{\alpha} \Vert\leq \Vert \VD^\alpha (F^{(k)})\Vert h^\alpha\,.
\end{equation}
The estimate \cref{eq:Fkdecay} is asymptotically equivalent to \cref{eq:Fdecay}.
This implies that modifying the transformation used to generated the isoparametric finite element
space does not affect its approximation properties
as long as $\Vert \VD^\alpha (F^{(k)})\Vert$ is moderate.

\begin{remark}
It is not strictly necessary to replace the transformation $T^{(k)}\circ F^{(0)}$
with its interpolant $F^{(k)}$. However, evaluating $T^{(k)}$ can be computationally
expensive and may not be supported natively in finite element software (in particular,
if evaluating $T^{(k)}$ involves a complicated formula).
On the other hand, as explained in the \cref{sec:optimization},
 $T^{(k)}$ lies in practice in a finite dimensional space.Therefore, the interpolation operator
$\Ci_h$ can be represented as a matrix, which can be used to significantly
speed up the computation of the update $dT^{(k)}$, as explained in \cref{sec:implementation}.
\end{remark}

\begin{remark}
Usually, the isoparametric transformation $F$ is chosen to be the identity on
elements of $\Delta_h(\pOmega)$ that do not share edges/faces with $\partial\pOmega$.
This particular choice of $F$ is made to decrease the computational cost of matrix/vector assembly,
and is not dictated by error analysis. In our approach, the function $F$ will generally differ from the
identity even in the interior of the domain.
\end{remark}

\begin{remark}
In \cite{HiPa15}, the authors suggest to purse shape optimization in Lagrangian coordinates by
reformulating shape optimization problems as an optimal control problem on the initial domain. The
resulting method is formally equivalent to the one presented in this work, but implies hard-coding of
geometric transformation into shape functionals and PDE-constraints (which is problem dependent),
and requires the derivation of Fr\'{e}chet derivatives in Lagrangian coordinates
(which are usually not considered in the shape optimization literature).
In contrast, the approach presented in this work exploits the fact that these geometric transformations
are included in standard finite element software, and allows the use of formulae for Fr\'{e}chet derivatives
that are already available in the literature.
\end{remark}

\section{Implementation aspects}\label{sec:implementation}
The previous sections consider different discretization aspects of the 
shape optimization problem \cref{eq:shapeoptproblem}.
\Cref{sec:optimization} introduces the finite dimensional space
$Q_N$ to construct the sequence of diffeomorphisms \cref{eq:Tsequence}.
\Cref{sec:isofem}, on the other hand, introduces the finite dimensional space $(V(\pOmega^0))^d$
to approximate the solution of \cref{eq:stateconstr} by means of isoparametric finite elements.

There are conflicting demands on the choice of these two finite dimensional spaces.
On the one hand, employing the same discretization
based on piecewise linear Lagrangian finite elements
greatly simplifies the implementation in existing finite element libraries and may reduce
the execution time. On the other hand, a decoupled discretization facilitates enforcing
stability in the optimization process. For instance, the authors of \cite{AlPa06,GiPaTr16} suggest the
use of linear Lagrangian finite elements
built on two nested meshes: a coarser one to discretize the geometry and a
finer one to solve the state equation. They report that this reduces the presence of spurious oscillations
in the optimized shape. 

A decoupled discretization may also be required
if one aims at using higher-order approximations of the state variable $u$. To elucidate this,
recall that the use of higher-order finite elements is motivated only if the
exact solution is sufficiently regular. More specifically,
isoparametric finite element solutions of degree $p$ converge as
$\Co(h^p)$ in the energy norm provided that the exact solution satisfies $u\in H^{p+1}(\Omega)$,
whereas the convergence rate deteriorates if $u$ is less regular \cite[Thm 3.2.1]{Ci02}
or if isoparametric finite elements are replaced by standard affine-parametric finite
elements \cite[Rmk 4.4.4 (ii)]{Ci02}.

It is virtually impossible to prescribe universal and sharp rules 
that ensure that the solution of the state constraint remains
sufficiently regular during the optimization process,
but elliptic regularity theory can provide some guidelines.
Assuming the problem data is sufficiently smooth, the solution of a linear elliptic Dirichlet BVP is
$H^{s}$-regular \cite[Def. 7.1]{Br07} when $\Omega$ has a $C^s$-boundary \cite[Thm 8.13]{GiTr01} (see also \cite{Gr11}
for an extensive treatment of elliptic regularity theory).
Therefore, it might be desirable to employ sufficiently regular transformations, so that
the regularity of the domain is preserved during the optimization process.
In this case, typical isoparametric Lagrangian finite elements are not a good choice for
$Q_N$ because they only allow $W^{1,\infty}$ piecewise polynomial representations of 
the domain transformations. The natural alternative is to employ multivariate B-splines
of degree $\tilde{p}$ \cite[Def. 4.1]{Ho03}, which are piecewise polynomials with compact support 
and are both $W^{\tilde{p},\infty}$ and $C^{\tilde{p}-1}$-regular.

For these reasons, we focus on the more general case of a decoupled discretization of
$(V(\pOmega^0))^d$ and $Q_N$ and discuss the implementation details for the following
simple optimization algorithm, which covers all fundamental aspects of shape optimization.

\subsection*{Minimal shape optimization pseudo-code}
\begin{enumerate}
\item initialize, then, for $k\geq 0$:
\item compute the state $u$ and evaluate $J$; stop if converged, otherwise
\item compute the update $dT$ solving \cref{eq:dTQn}
\item choose $s$ such that $T + sdT\circ T$ is feasible and $J$ is minimal
\item update $T$ and go back to step 2.
\end{enumerate}

\subsection*{Step 1}
First, we construct the finite element space
$(V_h(\pOmega^0))^d\coloneqq\mathrm{span}\left\{\Vv_i\right\}_{i=1}^M$
and store the coefficient vector $\Vf^{(0)}\in\bbR^M$ of the transformation
$F^{(0)}\in(V_h(\pOmega^0))^d$,
which maps $\pOmega^0$ onto (an approximation of) $\Omega^0$.
Then, we construct the space $Q_N\coloneqq\mathrm{span}\left\{\Vq_i\right\}_{i=1}^N$,
initialize the vector field $T^{(0)}\in Q_N$ to the identity, and store its coefficients in the vector
$\Vt^{(0)}\in\bbR^N$.
Finally, we store the matrix representation $\VI_h$ of the interpolation operator
$$\Ci_h : Q_N \to (V_h(\pOmega^0))^d\,.$$ The matrix $\VI_h$ is sparse if the basis functions
$\left\{\Vv_i\right\}_{i=1}^M$ and $\left\{\Vq_i\right\}_{i=1}^N$ have (small) compact support.

\subsection*{Step 2}
First, we compute the coefficients of $F^{(k)}\coloneqq\Ci_h(T^{(k)}\circ F^{(0)})$.
This is done by computing
\begin{equation}\label{eq:Fkalg}
\Vf^{(k)} = \Vf^{(k-1)}+\VI_h \Vd\Vt^{(k-1)},
\end{equation}
as will be justified in the next step.
Then, we approximate $u_{\Omega^{(k)}}$ by means of isoparametric
finite elements and evaluate $\Cj$ on the domain $\Omega^{(k)}$. If the convergence
criteria are not satisfied, we proceed further and compute an update of $T^{(k)}$.

\subsection*{Step 3}
First, we have to assemble the load vector
$\Vd\tilde{\VJ}^{(k)}_{\Vq}\coloneqq\{dJ(T^{(k)}; \Vq_i)\}_{i=1}^N$.
This can be computationally expensive because $dJ$ depends
on $u$, which is approximated with a finite element function
and lives on a finite element mesh. Therefore, to evaluate $dJ$
it is necessary to loop through the finite element
mesh. Although one loop is sufficient if one evaluates the contribution of each cell for all basis
functions in $Q_N$, evaluating these functions can be computationally expensive
and may require extensive modifications
of finite element software. Therefore, it may be desirable to employ a different
strategy, which we detail below.

Let $\{\Vv^{(k)}_i\}_{i=1}^M$ denote the isoparametric basis of $(V_h({\Omega}^{(k)}))^d$.
The vector
\begin{equation}\label{eq:dJisop}
\Vd\tilde{\VJ}^{(k)}_{\Vv^{(k)}}\coloneqq\{dJ(T^{(k)}; \Vv^{(k)}_i)\}_{i=1}^M\,,
\end{equation}
can be assembled efficiently with existing software, because
the basis functions $\{\Vv^{(k)}_i\}$ are generally included in finite element software
(whilst $\{\Vq_i\}$ may not be).
Interestingly, the product of the transpose of the interpolation matrix $\VI_h^\top$ with
\cref{eq:dJisop} can be interpreted as the approximation
\begin{equation}\label{eq:ITdJ}
\VI_h^\top\Vd\tilde{\VJ}^{(k)}_{\Vv^{(k)}}\approx 
\{dJ(T^{(k)};\Vq_i \circ (F^{(0)})^{-1}\circ(T^{(k)})^{-1})\}_{i=1}^N \,,
\end{equation}
where the right-hand side corresponds to the evaluation of $dJ(T^{(k)};\cdot)$ on functions
that move along with the domain transformation (see \cref{fig:transpqs}).
To explain the nature of the approximation in \cref{eq:ITdJ}, we introduce a new finite dimensional space
\begin{equation}\label{eq:newbasis}
Q_N^{(k)} \coloneqq\mathrm{span}\{\Vq^{(k)}_i\coloneqq \Vq_i \circ (F^{(0)})^{-1}\circ(T^{(k)})^{-1}\}_{i=1}^N\,.
\end{equation}
\begin{figure}[htb!]
\centering
\includegraphics[width=0.3\linewidth]{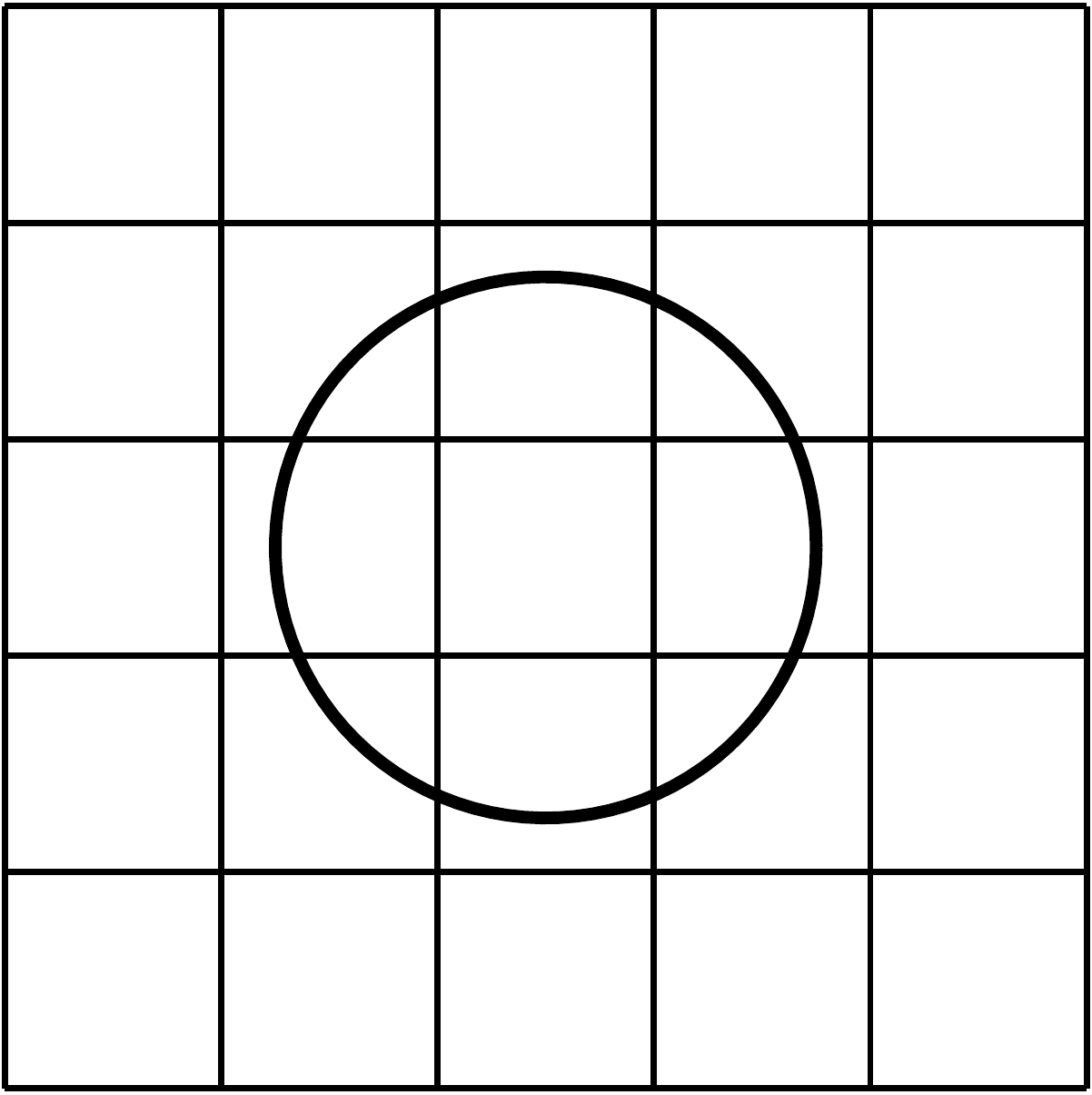}
\hspace{0.1cm}
\includegraphics[width=0.3\linewidth]{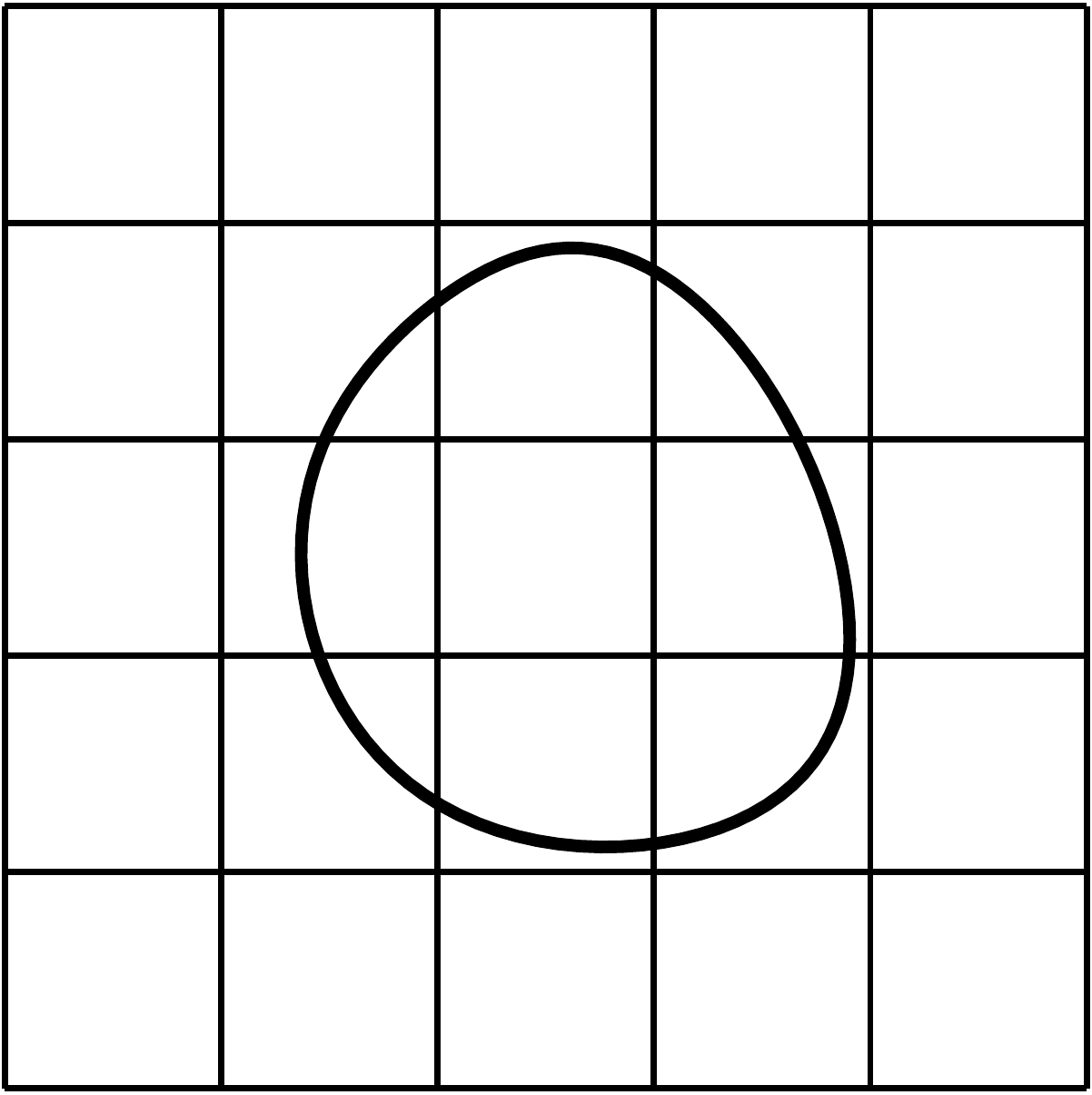}
\hspace{0.1cm}
\includegraphics[width=0.3\linewidth]{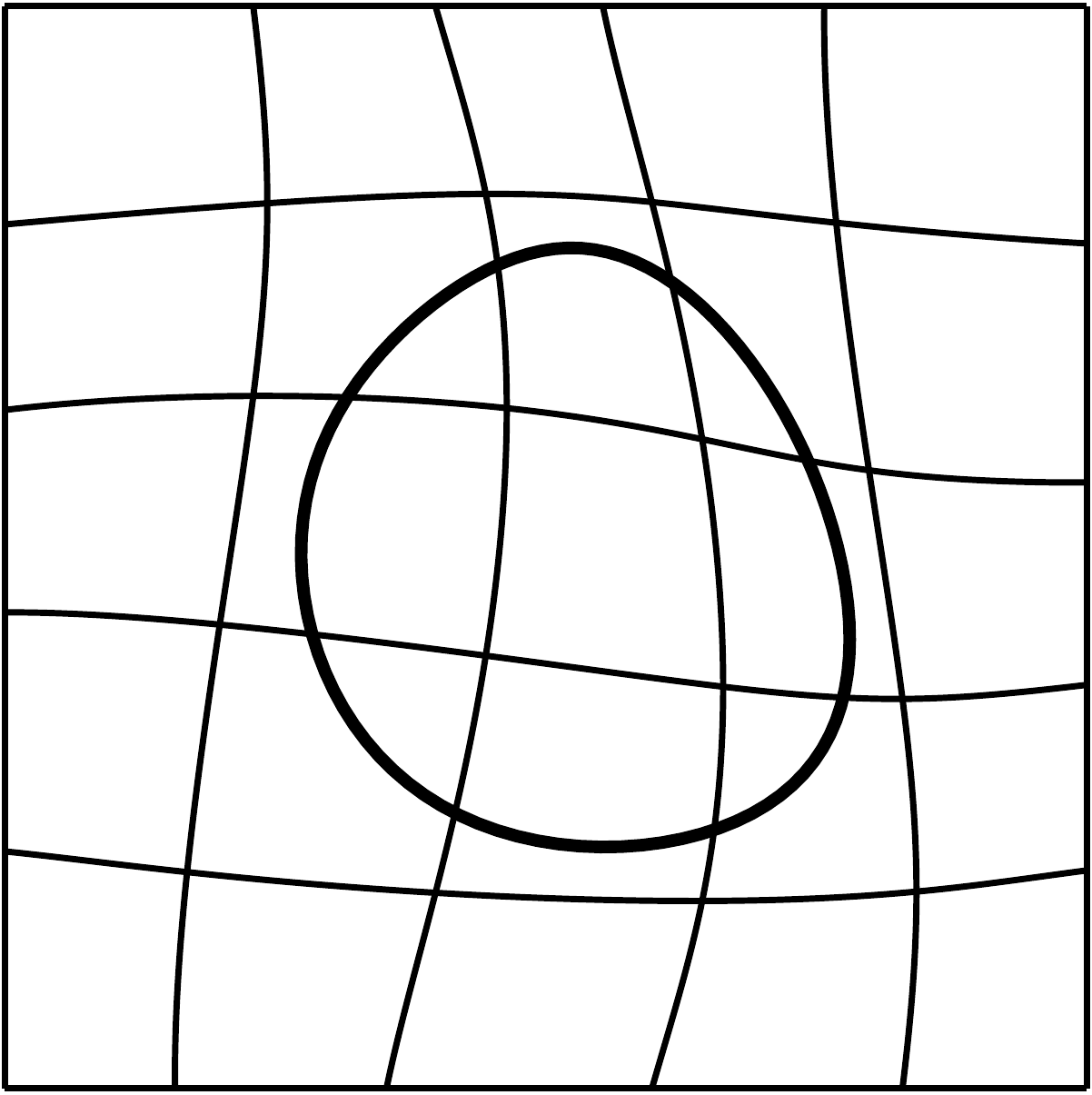}
\caption{Graphical example of a domain $\Omega^0$, a discretization $Q_N$, a perturbed
domain $\Omega^{(1)}$, and the discretization $Q_N^{(1)}$.
We assume that $Q_N$ is constructed on a grid.
\emph{Left:} initial configuration.
\emph{Center:} the domain is deformed according to a transformation $F^{(1)}$
whereas the basis functions $\Vq_i$'s of $Q_N$ are constructed on the initial regular grid.
\emph{Right:} both the domain and the grid evolve according to $F^{(1)}$; 
the basis functions $\Vq^{(1)}_i$'s of $Q_N^{(1)}$ are constructed on the curved grid
via pullback.}
\label{fig:transpqs}
\end{figure}
The space $Q_N^{(k)}$ arises naturally if
one considers shape optimization in Lagrange coordinates \cite{HiPa15}, and satisfies
$Q_N^{(k)}\subset W^{1,\infty}(\bbR^d,\bbR^d)$. Compared to $Q_N$,
this new space $Q_N^{(k)}$ has a great computational advantage: the previously 
computed matrix $\VI_h$ corresponds to the representation of
the interpolation operators
\begin{equation}
\Ci_h:Q_N \to (V_h(\Omega^0))^d
\end{equation}
as well as to the representation of
\begin{equation}\label{eq:IntOpk}
\Ci_h^{(k)}:Q_N^{(k)} \to (V_h({\Omega}^{(k)}))^d
\end{equation}
with respect to the basis $\{ \Vq^{(k)}_i\}_{i=1}^N$ and $\{\Vv^{(k)}_i\}_{i=1}^M$.
This implies that  the vector $\Vd\tilde{\VJ}^{(k)}_{\Vq^{(k)}}$ (the right-hand side of \cref{eq:ITdJ})
can be assembled more easily and quickly than $\Vd\tilde{\VJ}^{(k)}_{\Vq}$. On top of that, 
\cref{eq:IntOpk} motivates the interpretation \cref{eq:ITdJ}. Note that for given test cases,
the asymptotic rate of convergence in this approximation can be explicitely computed.


Once the load vector for the update step has been assembled, we need to assemble the stiffness matrix.
In principle, this should be done as well with respect to the new discretization
\cref{eq:newbasis}. Alternatively, one can redefine the inner
product of $\Cx$ so that
$$\big( \Vq_i,\Vq_j \big)_\Cx = \big( \Vq_i^{(k)},\Vq_j^{(k)} \big)_{\Cx^{(k)}}
\quad \text{for all } i,j=1,\dots,N\,,$$
and the stiffness matrix can be computed once and for all in the initialization step.

Finally, one computes the vector $\Vd\Vt^{(k)}$, which contains the coefficients of the
series expansion of $dT^{(k)}$ with respect to the basis of \cref{eq:newbasis}. Computing
$\VI_h \Vd\Vt^{(k)}$ returns the coefficients of the interpolant of $dT^{(k)}$ in
$(V_h({\Omega}^{(k)}))^d$. Due to the isoparametric nature of its basis $\{\Vv^{(k)}_i\}_{i=1}^M$,
these coefficients equal those resulting from interpolating $dT^{(k)}\circ F^{(k)}$ onto
$(V(\pOmega^0))^d$. This explains why \cref{eq:Fkalg} is the algebraic counterpart of \cref{eq:Fk}.

\subsection*{Step 4-5}
The simplest way to compute $s$ is by line search.
In this case, we have to evaluate $J$ on $T^{(k)}+sdT^{(k)}$ for various $s$.
In light of \cref{eq:Fkalg}, this means computing the state variable $u$ using
isoparametric finite elements whose coefficients are
$$\Vf^{(k+1)}_s = \Vf^{(k)}+s\VI_h \Vd\Vt^{(k)}\,,$$
and choosing $s$ such that the value of $J$ is minimal.
Of course, one has to enforce admissibility of $T^{(k)}+sdT^{(k)}$.
By \cref{thm:Hadamard}, it is sufficient to verify that the minimum of
\begin{equation}\label{eq:detDFk}
\det(\VD( F^{(k)}+s\Ci_h(dT^{(k)})))\approx \det(\VD (T^{(k)}+sdT^{(k)}))
\end{equation}
remains positive. However, note that small values of \cref{eq:detDFk}
may negatively affect the ellipticity constant of the BVP \cref{eq:stateconstr}, 
which in turn negatively affects the constant of finite element error estimates.
Finally, one rescales $dT^{(k)}\coloneqq sdT^{(k)}$, set
$T^{(k+1)}=(\Ci + \mathrm{d}T^{(k)})\circ(T^{(k)})$, and goes back to step 2.

\section{Numerical experiments}\label{sec:numexp}
We split our numerical investigations in two parts. In the first one,
we consider a PDE-constrained shape optimization problem that
admits stable minimizers. We use this test case to investigate the
approximation properties of the algorithm presented in \cref{sec:implementation}
for different discretizations of control and state variables. In the second part,
we test our approach on more challenging shape optimization problems for which analytical
solutions are unavailable.
The numerical results are obtained with a code based on
the finite element library Firedrake
\cite{petsc-user-ref,petsc-efficient,Dalcin2011,Rathgeber2016,Luporini2016,Chaco95,MUMPS01,MUMPS02}.

\subsection{Bernoulli free-boundary problem}
\label{sec:numexpBernoulli}
We consider the Dirichlet BVP
\begin{equation}\label{eq:BPdir}
- \Delta u = 0\;  \text{ in } \Omega\,,\; \quad u = \uin\quad  \text{ on } \dOin\,, 
\;\quad u = \uout\;  \text{ on } \dOout\,,
\end{equation}
stated on the domain $\Omega\subset \bbR^2$ depicted in \cref{fig:BernoulliDomain}.
The goal is to find the shape of $\dOin$ so that
the Neumann trace $\dn{u}$ is equal to a prescribed function $g$ on $\dOin$. 
For the sake of simplicity, we assume that the Dirichlet
data $\uin$ and the Neumann data $g$ are constants, and that only
$\dOin$ is ``free'' to move.

\begin{figure}[!htb]
    \centering
    \includegraphics{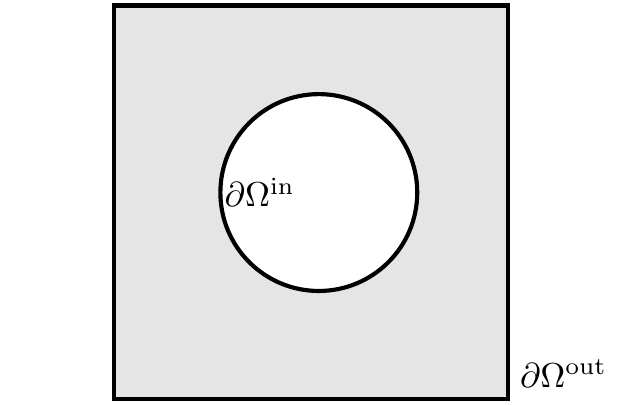}
    \caption{Computational domain for the Dirichlet BVP \eqref{eq:BPdir}.
        The external boundary $\dOout$ is a square centered at the origin with a corner in (1,1).
        The internal boundary $\dOin$ is optimized to achieve $\dn{u}\vert_{\dOin} = g$.
        }
    \label{fig:BernoulliDomain}
\end{figure}

This Bernoulli free boundary problem can be reformulated as the following
shape optimization problem \cite{EpHa06b}
\begin{equation}\label{eq:benchmark}
\inf_{\Omega \in \Uad} \Cj(\Omega, u)=\int_\Omega \nabla u\cdot\nabla u + g^2 \DX\,,
\quad \text{subject to} 
\left\{\begin{array}{rll}
-\Delta u &= 0 &\text{in }\Omega\,,\\
u &= u_\textrm{in} & \text{on } \partial\Omega^{\mathrm{in}}\,,\\
u &= u_\textrm{out} & \text{on } \partial\Omega^{\mathrm{out}}\,,
\end{array}\right.
\end{equation}
whose Fr\'{e}chet derivative reads \cite{HiPa15}
\begin{equation}\label{eq:BerDer}
dJ(T;\Ct) = \int_{T(\Omega_0)} 
\Div \Ct (\nabla u_{T(\Omega_0)}\cdot\nabla u_{T(\Omega_0)}+ g^2 )
-\nabla u_{T(\Omega_0)}\cdot (\VD\Ct +\VD\Ct^\top)\nabla u_{T(\Omega_0)}\, \DX\,.
\end{equation}

In \cite{EpHa06b}, the authors have studied this shape optimization problem \cref{eq:benchmark}
in detail and, performing shape analysis in polar coordinates, have shown that
the shape Hessian is both continuous and coercive (when restricted to normal perturbations)
in the $H^{1/2}(\partial\Omega)$-norm. For this reason, minimizers of \cref{eq:benchmark} are stable.

To construct a test case for our numerical simulations, we set the optimal shape
of $\dOin$ to be a circle centered at the origin with radius $0.4$.
For such a choice of $\dOin$, the function (expressed in polar coordinates)
$$u(r, \varphi) = \ln(0.4) - \log(r)$$
satisfies the Dirichlet BVP \eqref{eq:BPdir} with $u_\textrm{in} = 0$
and $u_\textrm{out} = u$. The Neumann data on the interior boundary is $g=2.5$.
The value $J_\mathrm{min}$ of the misfit functional in the optimal shape is approximatively \texttt{28.306941614057237}.
This value has been computed with quadratic isoparametric finite elements on a
sequence of nested meshes; the relative error between the value of the misfit functional
computed on the last and on the second last mesh is approximately $6\cdot 10^{-11}$.

As initial guess, we set $\dOinzero$ to be a circle of radius 0.5 centered at (0.04,0.05).
Note that we have repeated these numerical experiments with other 3 choices for the initial guess
$\dOinzero$ and have obtained similar results. These alternative initial guesses are:
a circle of radius 0.47 centered at (0.07, 0.03), a circle of radius 0.55 centered at (-0.1, 0), 
and a circle of radius 0.5367 centered at (-0.137, 0.03).

To discretize geometric transformations, we consider linear/quadratic/cubic tensorized Schoenberg
B-splines constructed on regular grids \cite{Sc15}.
These grids are refined uniformly (with widths ranging from $1.8\times2^{-1}$ to $1.8\times2^{-6}$)
and are contained in a square (the hold-all domain $D$) that is centered at the origin and has a corner at
(0.95, 0.95), so that $\dOout$ is not modified in the optimization process. Finite element approximations of
$u_{T(\Omega_0)}$ are computed with linear/quadratic isoparametric finite elements on a sequence
of 5 triangular meshes generated using uniform refinement in Gmsh \cite{GeRe09}. Note that
finer meshes are adjusted to fit curved boundaries.

The optimization is carried out by repeating the following simple procedure for a fixed number
of iterations: at every iteration, we
compute a $H^1_0(D)$-descent direction $dT$ by solving \cref{eq:dTQn} and choose the
optimization step $s\in\{0, 0.1, \dots, 1\}$ that minimizes $J(T+ sdT\circ{T})$.
Such a simple optimization strategy is sufficient for our numerical experiments,
although we are aware that it is not efficient. 
The development of more efficient optimization strategies in the context of shape optimization
is a current topic of research. In \cite{ScSiWe16}, the authors obtained promising results with
a BFGS-type algorithm based on a Steklov--Poincar\'{e} metric. We defer to future research
the numerical comparison of optimization strategies for shape optimization.

In \cref{fig:linesearchguess1}, we plot two steps of this simple optimization strategy.
Transformations are discretized with quadratic B-splines built on the fourth grid,
whereas the state $u_{T(\Omega_0)}$ is approximated with linear finite elements on the
second coarsest mesh. Qualitatively, we observe the expected behavior of a (truncated) linesearch.
The predicted-descent line is given by $J(\Ct)- dJ(\Ct, sdT)$, with
$s=0, 0.1, 0.2, 0.3$, and is tangential to $J(T+ sdT\circ{T})$ at $s=0$. This shows empirically that
formula \cref{eq:BerDer} is correct.
\begin{figure}[!htb]
\centering
\includegraphics{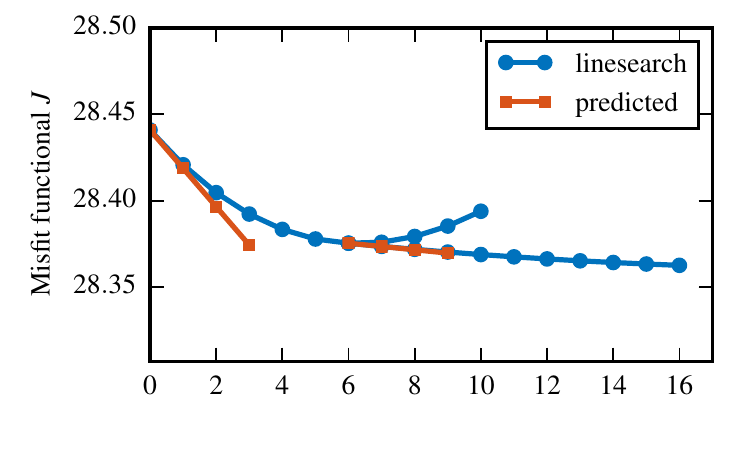}
\caption{Evolution of $J$ on two optimization steps. The second linesearch starts at the minimum
of the first one. The predicted descent is computed evaluating $dJ$ on the (selected) descent direction.}
\label{fig:linesearchguess1}
\end{figure}

Next, we investigate to which accuracy we can solve shape optimization problems.
In particular, we study the impact of the discretization (polynomial-)degree and
(refinement-)level used for the control and the state.
This is done systematically by keeping certain discretization parameters fixed
and varying the remaining ones.
To simplify the exposition, we associate the term \emph{grid} to the
discretization of the control (that is, of geometric transformations),
whereas the term \emph{mesh} refers to the discretization of the state).

First, we fix the control discretization to quadratic B-splines built on the finest grid.
For each finite element mesh (previously generated with Gmsh), we perform 101 steps of our
simple optimization strategy employing linear FE approximations of the state $u_{T(\Omega_0)}$.
Although not displayed,
the sequence of shapes always converges qualitatively to the optimum. For a quantitative comparison,
we store the minimum of $J$ for every linesearch and plot the absolute error with respect to 
$J_\mathrm{min}$ in \cref{fig:convergencehistory_linearfemguess1}
(we plot a convergence history for each mesh). Henceforth, we use the notation \texttt{Jerr}
to refer to this absolute error.
\begin{figure}[!htb]
\includegraphics{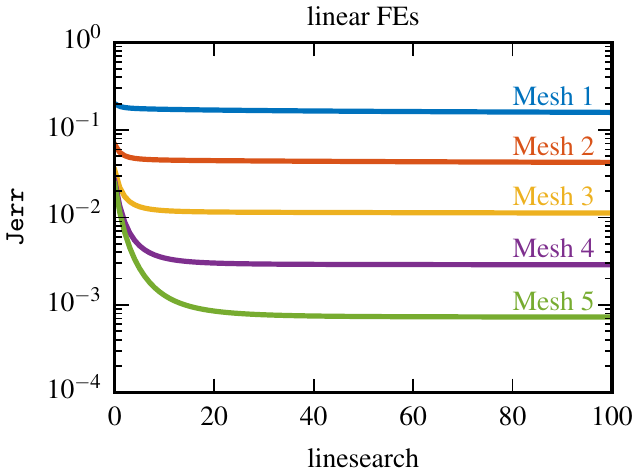}
\includegraphics{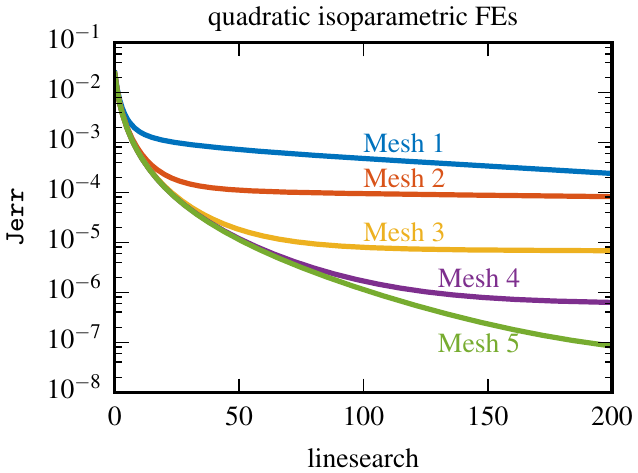}
\caption{Evolution of \texttt{Jerr} when the state variable is discretized with linear (\emph{left})
and quadratic (\emph{right}) FEs  on a different meshes constructed via uniform refinement.
Convergence lines saturate at different levels, which decay algebraically with respect to the mesh width.}
\label{fig:convergencehistory_linearfemguess1}
\end{figure}
We observe that the convergence history lines
saturate at different levels. In particular, the saturation level decays algebraically with respect
to the mesh width. To further investigate the impact of FE approximations on shape optimization,
we repeat this experiment with quadratic isoparamentric FEs.
Again, we observe algebraic convergence with respect to the mesh width, but at higher convergence rate
(note the difference in the y-axis scale). In order to reach the saturation level on finer meshes,
more optimization steps have to be carried out. This issue has been observed in \cite{HiPa15} as well, and 
is probably due to the simplicity of the optimization strategy. Before proceeding further, let us point out that
the saturation level worsens if quadratic isoparametric FEs are replaced by quadratic affine FEs;
see \cref{fig:convergencehistory_11vs12vs22femguess1}.
\begin{figure}[!htb]
\centering
\includegraphics{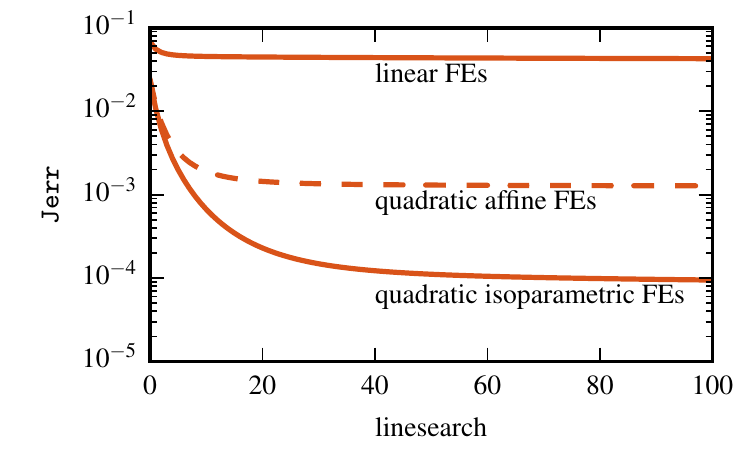}
\caption{Evolution of \texttt{Jerr} when the state variable is discretized with linear, quadratic affine,
and quadratic isoparametric FEs on the second mesh.
Quadratic affine FEs perform worse than their isoparametric counterpart.}
\label{fig:convergencehistory_11vs12vs22femguess1}
\end{figure}

In the previous experiments, we kept the discretization of transformations fixed. 
Now, we test different discretizations.
In \cref{fig:meshsplineconvergenceguess1}, we consider two different discretization degrees
of the control variable (linear/cubic B-splines).
For each of these discretization degrees, we consider 6 grids and 5 meshes.
The state is approximated once with linear and once with quadratic isoparametric FEs.
For each combination, we plot \texttt{Jerr} after 200 iterations.
We observe that both the discretization of the control and of the state
have an impact on \texttt{Jerr} (the algebraic decay with respect to the FE mesh width is conspicuous).
\begin{figure}[!htb]
\includegraphics{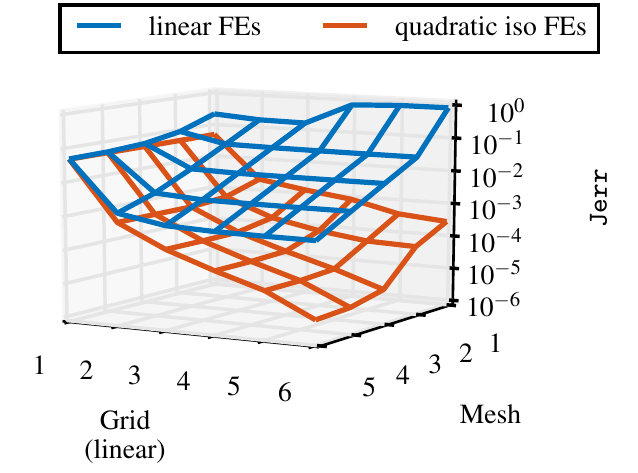}
\includegraphics{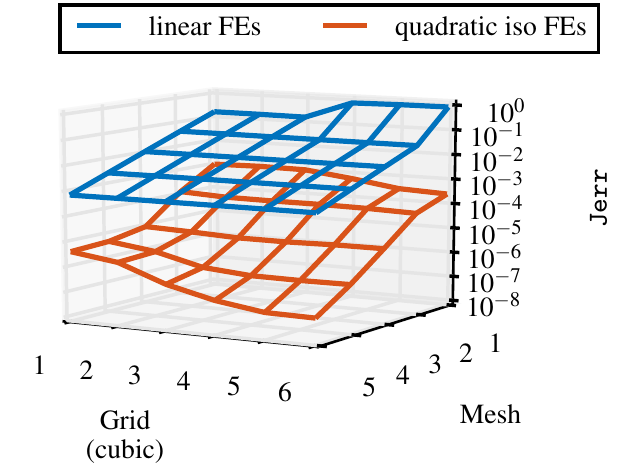}
\caption{Value of \texttt{Jerr} after 201 optimization steps for different combination of 
control and state discretization. The algebraic rate of convergence with respect to mesh refinement
and the benefits of higher-degree discretizations of transformations are clearly visible. }
\label{fig:meshsplineconvergenceguess1}
\end{figure}

In \cref{fig:splineconvergenceguess1} (left), we consider the finest level and highest degree
of the control discretization.
We plot \texttt{Jerr} (after 200 iterations) versus the FE mesh index and consider both linear and quadratic
isoparametric FE approximations of the state.
The algebraic rates of convergence for linear and quadratic FEs read $1.97$ and  $3.24$, respectively.
This rates are in line with our expectations because duality techniques can be employed
to prove superconvergence in the FE approximation of the quadratic functional $J$ \cite{HiPaSa15}.
\begin{figure}[!htb]
\includegraphics{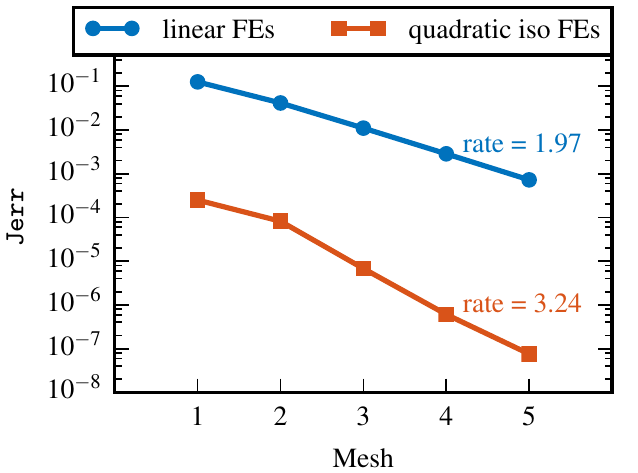}
\hfill
\includegraphics{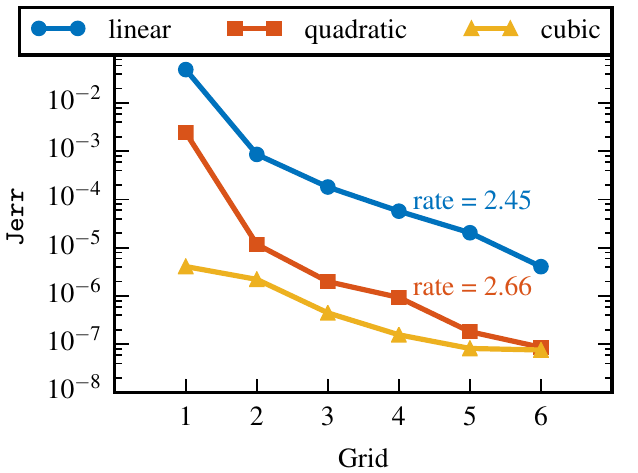}
\caption{Details from \cref{fig:meshsplineconvergenceguess1}:
\emph{Left:} convergence with respect to FE discretization of the state.
\emph{Right:} convergence with respect to B-spline discretization of the control.}
\label{fig:splineconvergenceguess1}
\end{figure}
\cref{fig:meshsplineconvergenceguess1} shows also that \texttt{Jerr} is almost entirely dominated
by the error in the state when transformations are discretized with cubic B-splines.
This is better highlighted in \cref{fig:splineconvergenceguess1} (right).
There, we fix the state discretization to quadratic isoparametric FEs on the
finest mesh and plot \texttt{Jerr} versus the grid index for linear, quadratic,
and cubic B-splines. For instance, grid refinement for cubic B-splines
has a very mild impact on \texttt{Jerr} (due to the FE approximation error of the state).
We observe that the algebraic convergence rate for grid refinement of linear B-splines is approximately
2.45. This rate is higher than our expectations (a perturbation analysis via the Strang lemma
would give rise to only linear convergence).

Finally, we consider the highest discretization level of the control and investigate whether its
discretization degree has an impact on the rate of convergence of \texttt{Jerr} with respect
to the FE discretization of the state.
Let us recall that the regularity
of the state on a perturbed domain depends, in principle, on the regularity of the domain.
When perturbed with less smooth transformations, the resulting domain may not
guarantee that the regularity of the state is preserved. This may have a negative impact on
the FE approximation. In \cref{fig:meshconvergence_impactofbsplineregularityguess1},
we plot \texttt{Jerr} versus FE mesh refinement (both for linear and quadratic FEs) when
transformations are discretized with linear B-splines.
When the state is approximated with quadratic isoparametric FEs, the control discretization error
is negligible only on much finer grids.
However, it seems that the FE convergence rate is not affected by the
the discretization degree of the control (in \cref{fig:splineconvergenceguess1}, left,
the control is discretized with cubic B-splines, and a similar convergence rate is observed).

\begin{figure}[!htb]
\centering
\includegraphics{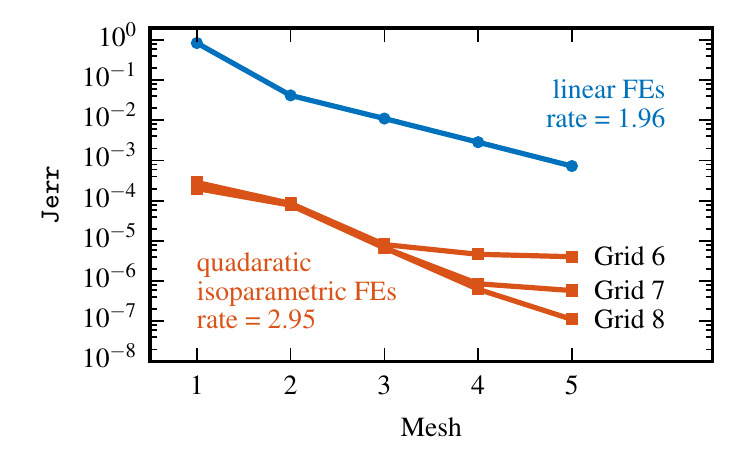}
\caption{Surprisingly, discretizing transformations with linear B-splines instead of cubic
does not affect the rate of convergence of \texttt{Jerr} with respect to FE discretization of the state
(compare with \cref{fig:splineconvergenceguess1}, left).}
\label{fig:meshconvergence_impactofbsplineregularityguess1}
\end{figure}

\subsection{Energy minimization in Stokes flow}
\label{subsec:stokes}
Next, we consider the shape optimization of a 2D obstacle $\Omega_o$ embedded in a viscous fluid $\Omega_c$
(see \cref{fig:stokescp}).
This problem has been thoroughly investigated in \cite{ScSi15}.
The state variables are the velocity $\Vu$ and the pressure $p$ of the fluid. These are
governed by Stokes' equations, whose weak formulation reads:
find $\Vu\in H^1(\Omega_c;\bbR^2)$ and $p\in L^2_0(\Omega_c)=\{p \in L^2(\Omega_c) : \int_{\Omega_c} p \DX = 0\}$ such that
$\Vu\vert_{\Gamma_{\text{in}}} = \Vg,\ u\vert_{\Gamma_{wall}} = 0$ and 
\begin{equation} \label{eq:stokesweak}
	\int_{\Omega_c} \sum_{i=1}^2\nabla u_i\cdot \nabla v_i - p \Div \Vv  + q\Div \Vu \DX  = 0
\end{equation}
for all $\Vv\in H^1(\Omega_c;\bbR^2)$ and $q\in L^2_0(\Omega_c)$ such that $\Vv\vert_{\Gamma_{\text{in}}} = 0$ and $\Vv\vert_{\Gamma_{\text{wall}}} = 0$.
It is known that \cref{eq:stokesweak} admits a unique solution if the computational domain
$\Omega_c$ is Lipschitz \cite{GiRa86, Te01}.

\begin{figure}[!htp]
    \centering
    \includegraphics{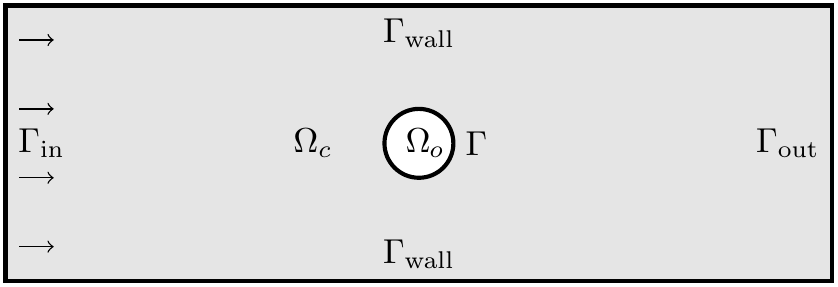}
    \caption{Computational domain for the Stokes BVP \eqref{eq:stokesweak}.
        The external boundary comprises: the inflow boundary $\Gamma_\mathrm{in}$, the outflow
        boundary $\Gamma_\mathrm{out}$, the obstacle boundary $\Gamma$,
        and the wall boundary $\Gamma_\mathrm{wall}$. The obstacle boundary is a circle centered
        at the origin with radius $0.5$, whereas the remaining boundary is a rectangle with width
        $12$ and height $5$.}
    \label{fig:stokescp}
\end{figure}

The energy dissipated in the fluid due to shear forces is given by 
\begin{equation}\label{eq:stokesJ}
    \Cj(\Omega_c, \Vu, p) = \int_{\Omega_c} \sum_{i=1}^2 \nabla u_i^\top \nabla u_i \DX \equiv\int_{\Omega_c}  \lVert\VD\Vu\rVert_{F}^2 \DX, \,,
\end{equation}
and its shape derivative is given by\footnote{We believe that this volume based formula
(which can be computed with \cite{Sc16})
is already known to the shape optimization community, although we did not manage to
find it explicitly in available publications.}
\begin{multline} \label{eq:stokedJ}
d\Cj(T, \Ct)= \int_{T(\Omega_c)}  
\sum_{i=1}^2 \left( \nabla u_i^\top \nabla u_i \Div\Ct-\nabla u_i^\top(\VD{\Ct}^\top + \VD{\Ct} ) \nabla u_i\right)\\
+ 2\tr(\VD \Vu \VD{\Ct} )p - 2p\Div \Vu \Div\Ct\DX\,.
\end{multline}

Minimizing \eqref{eq:stokesJ} subject to \eqref{eq:stokesweak} alone is problematic because
energy dissipation can be reduced by shrinking or removing the obstacle. 
Therefore, we introduce two additional constraints: we require the area and the 
barycentre of the obstacle to remain constant. 
Similarly to  \cite{ScSi15}, we define the functionals
\begin{equation}
	\begin{aligned}
		\Ca(T) &\coloneqq \int_{T(\Omega_c)} 1 \DX - \int_{\Omega_c} 1 \DX\,, \\
		\Cb_i(T) &\coloneqq \int_{T(\Omega_c)} x_i \DX - \int_{\Omega_c} x_i \DX , \text{ for } i=1,2\,. 
	\end{aligned}
\end{equation}
For the sake of simplicity, we enforce the constraints $\Ca(T)=0$ and $\Cb_i(T)=0$, $i=1,2$ 
onto the shape optimization problem in the form of penalty functions, that is, we replace the functional
\eqref{eq:stokesJ} with
\begin{equation}\label{eq:stokespenalized}
	\Cj_p(T) = \Cj(T) + \frac{\mu_0}{2} \Ca^2(T) + \sum_{i=1}^2 \frac{\mu_i}{2} \Cb_i^2(T)\,,
\end{equation}
where $\Ca^2(T)\coloneqq(\Ca(T))^2$, $\Cb_i^2(T)\coloneqq (\Cb_i(T))^2$, and
 $0\leq \mu_i\in\bbR$, i=0,1,2.
The shape derivatives of the squared constraints are given by 
\begin{equation}
	\begin{aligned}
		d{A^2}(T, \Ct) &= 2 \Ca(T) \int_{T(\Omega_c)} \Div \Ct\DX\,, \\
		d{B_i^2}(T, \Ct) &= 2 \Cb_i(T) \int_{T(\Omega_c)} \Div (x_i \Ct) \DX, \text{ for } i=1, 2\,.
	\end{aligned}
\end{equation}

The state variables are discretized with Taylor-Hood P2-P1 finite elements on
a triangular mesh. This discretization is stable \cite{ElSiWa14}.
The resulting linear system can then be solved using GMRES and a block-diagonal preconditioner 
based on the stiffness matrix and on the mass matrix for the velocity- and the pressure-block,
respectively \cite{ElSiWa14}. 
\begin{figure}[htbp]
    \begin{center}
        \includegraphics[width=0.445\textwidth]{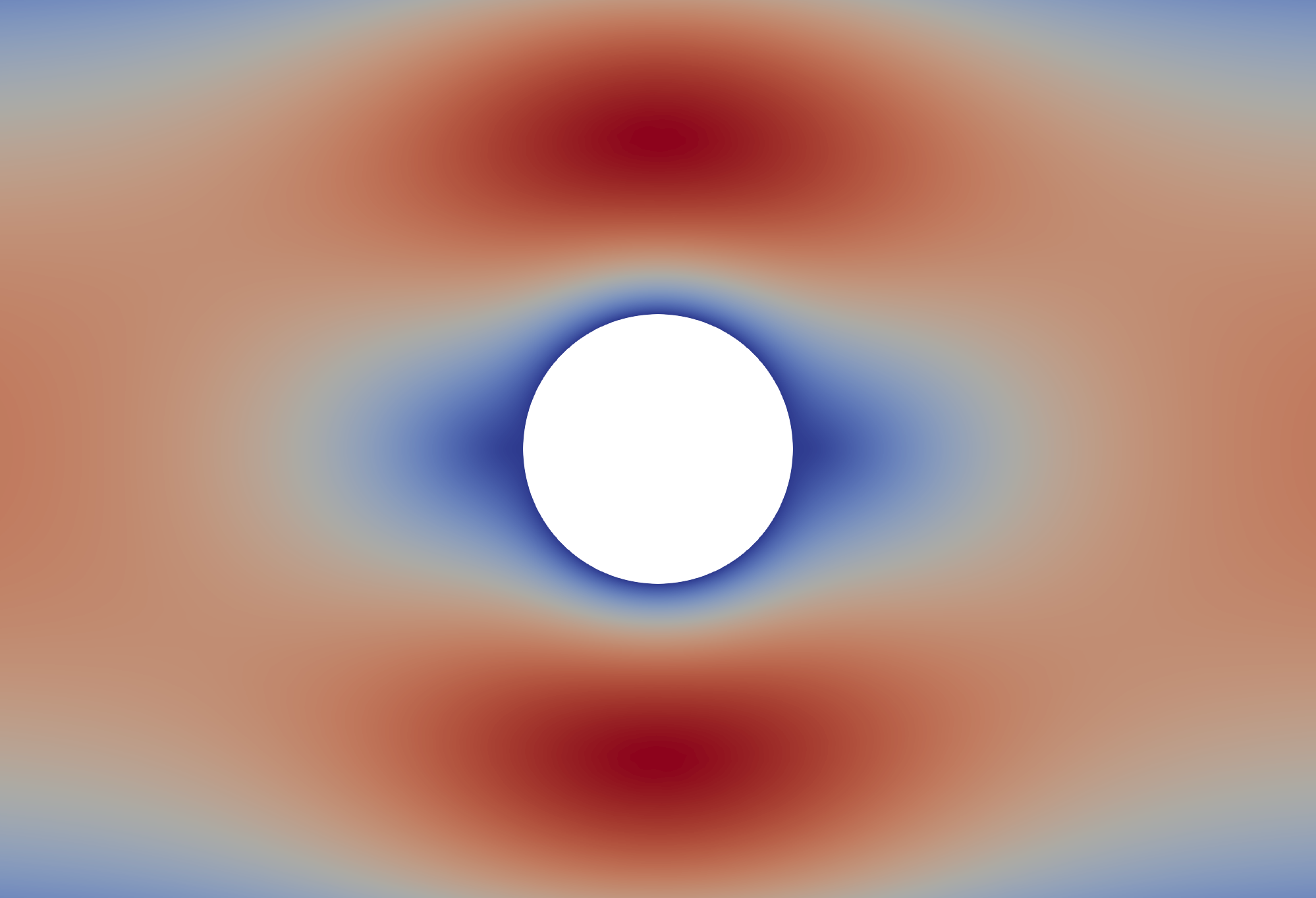} 
        \includegraphics[width=0.445\textwidth]{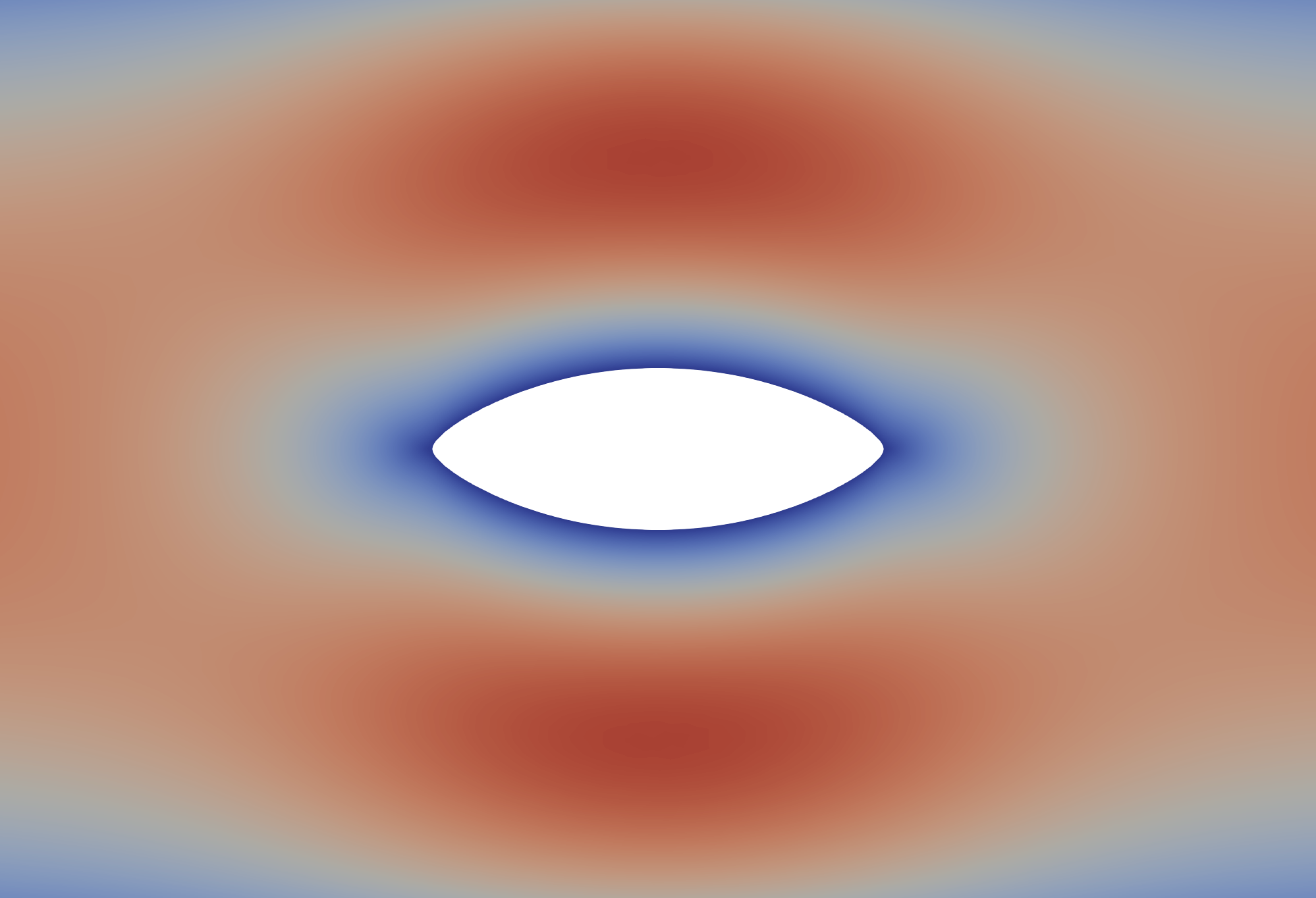} \\
        \includegraphics{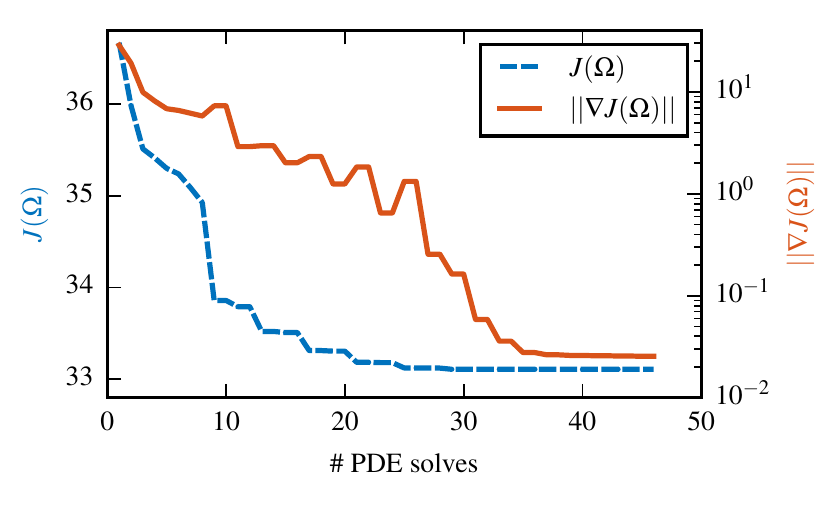}
    \end{center}
    \caption{\emph{Left:} initial shape. \emph{Right:} optimized shape using quadratic B-splines. The colour indicates the magnitude of the velocity field. \emph{Bottom:} convergence history; the gradient is measured using the $H^1(D)$-norm.}
    \label{fig:stokes-shapes}
\end{figure}
The control is discretized with cubic B-splines on a rectangular
grid that covers the entire channel.
We can see that the optimized shape in \cref{fig:stokes-shapes} is qualitatively similar to the one obtained by \cite{ScSi15},
although the shape optimization algorithm used in \cite{ScSi15} relies on the
boundary based formulation of \cref{eq:stokedJ}. 
For this example we have used a simple steepest descent with line search; the number of optimization steps can be drastically reduced employing
more sophisticated optimization algorithms \cite{ScSi15}.
\subsection{Compliance minimization under linear elasticity}
We conclude this section on numerical experiments with another classical example from shape optimization:
compliance minimization of a cantilever subject to a given load (see \cref{fig:elasticitycp}).
\begin{figure}[!htp]
    \centering
    \includegraphics{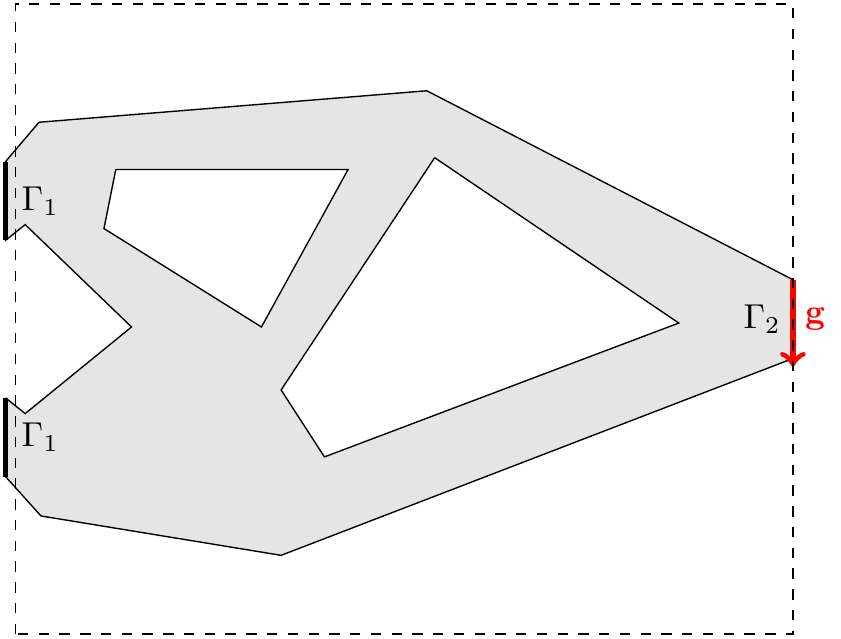}
    \caption{Computational domain for the elasticity BVP \eqref{eq:elasticityweak}.
On the two segments $\Gamma_1$ we impose homogeneous Dirichlet boundary conditions.
On $\Gamma_2$  we impose $\Vg=(0,-1)^\top$ as Neumann boundary condition. The dashed rectangle indicated the hold-all domain $D$. We exclude the part close to the left boundary, as singularities occur in the solution to the PDE close to the corners.}
    \label{fig:elasticitycp}
\end{figure}
The structural behaviour of the cantilever is modelled by linear elasticity.
In particular, we consider the following variational problem: find $\Vu\in H^1(\Omega;\bbR^2)$ such that
$\Vu\vert_{\Gamma_1} = 0$ and 
\begin{equation} \label{eq:elasticityweak}
    \int_{\Omega} (A e(\Vu)):\nabla \Vv \DX  - \int_{\Gamma_2} \Vg \cdot \Vv \dS = 0
\end{equation}
for all $\Vv\in H^1(\Omega;\bbR^2)$ with $\Vv\vert_{\Gamma_1} = 0$.
In \cref{eq:elasticityweak}, the symbol $:$ denotes the Frobenius inner product of matrices,
$e(\Vu) = \sym(\nabla \Vu)= \frac{1}{2}(\nabla \Vu + \nabla \Vu^\top)$ denotes the
strain tensor, and $A$ encodes the Hookes' law of the material.
It is well known that \cref{eq:elasticityweak} admits a unique stable solution for compatible data
$\Vg$ \cite[Page 22]{Ci02}.
In this numerical experiment, we minimize the compliance
\begin{equation}\label{eq:elasticityJ}
    \Cj(\Omega, \Vu) = \int_{\Omega} (A e(\Vu)): e(\Vu) \DX\,,
\end{equation}
whose shape derivative reads (the formula can be derived following closely \cite[Thm 7.7]{St14})
\begin{equation}
    \begin{aligned}
        d\Cj(T,\Ct) =& \int_{T(\Omega)} A\sym(\VD \Vu \VD \Ct):\sym(\VD \Vu)\DX\\
        & +\int_{T(\Omega)} A\sym(\VD \Vu):\sym(\VD \Vu\VD \Ct)\DX\\
        & +\int_{T(\Omega)} A\sym(\VD \Vu):\sym(\VD \Vu)\Div \Ct\DX\,.
    \end{aligned}
\end{equation}
Similarly to \cref{eq:stokespenalized}, we enforce a volume constraint by adding a penalty function
to \cref{eq:elasticityJ}.
We discretize the state variables with piecewise linear Lagrangian finite elements on a triangular mesh;
the control is discretized with cubic B-splines on a rectangular grid covering the cantilever.
\begin{figure}[htbp]
    \begin{center}
        \includegraphics[width=0.49\textwidth]{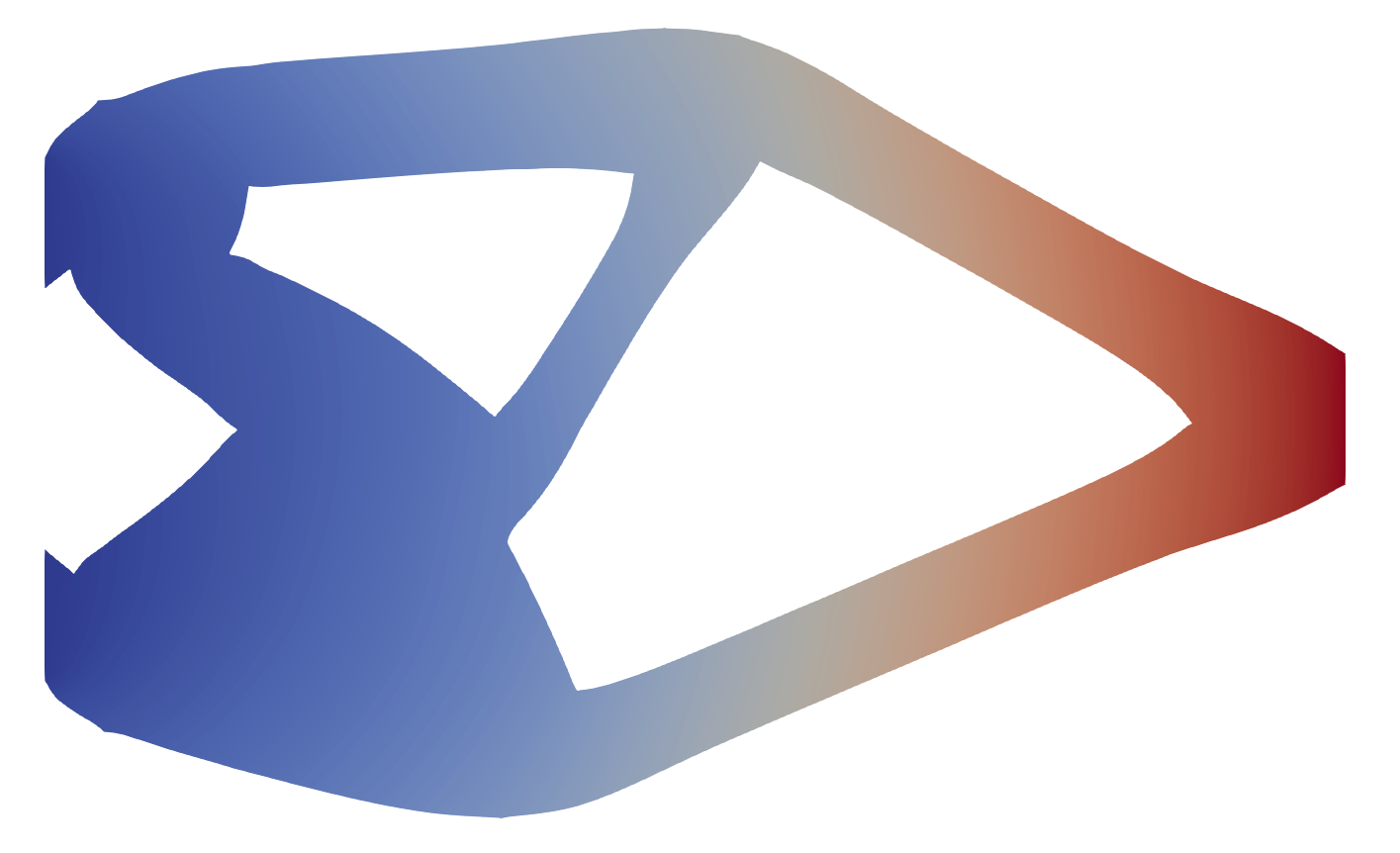} 
        \includegraphics[width=0.49\textwidth]{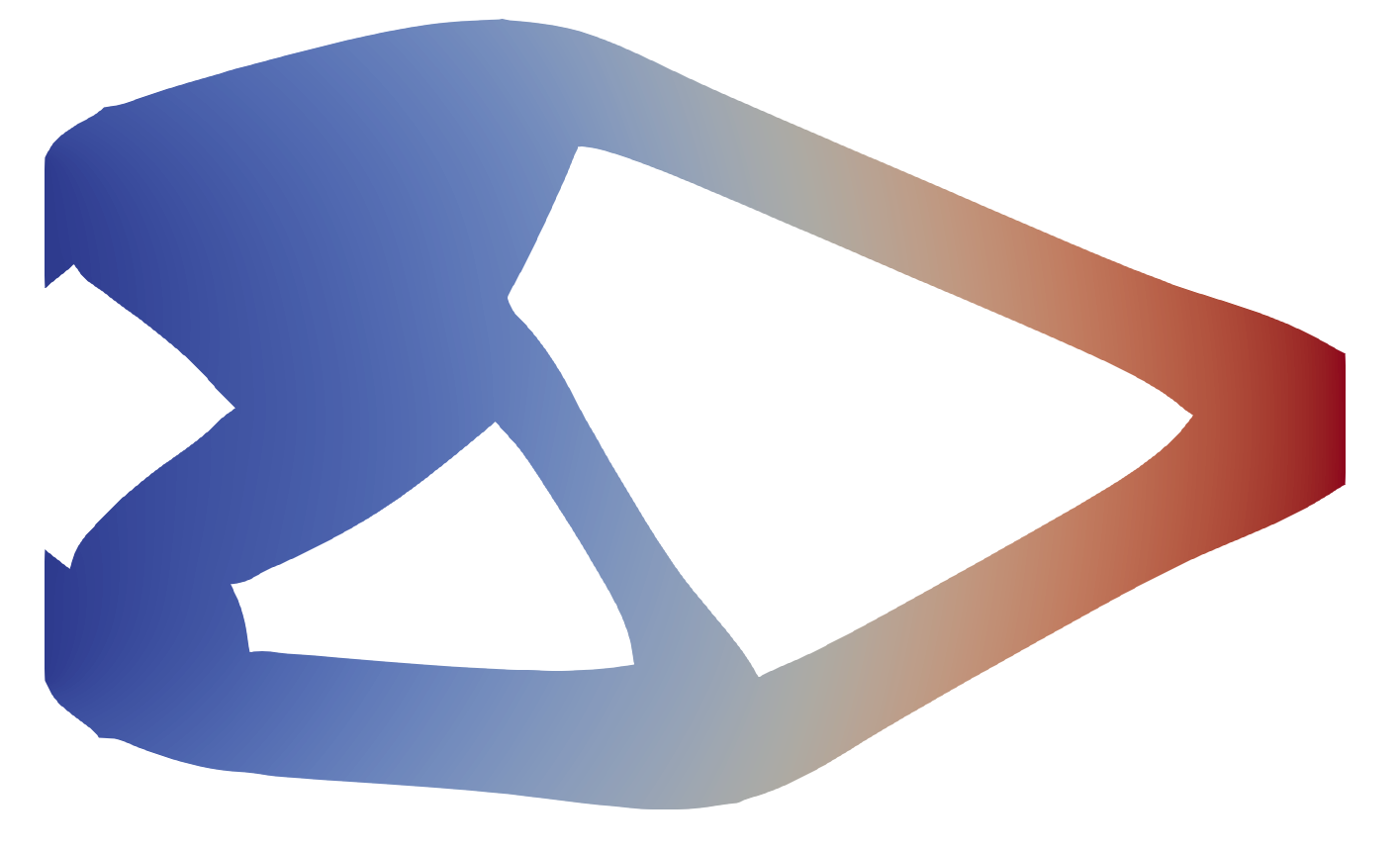} 
        \includegraphics{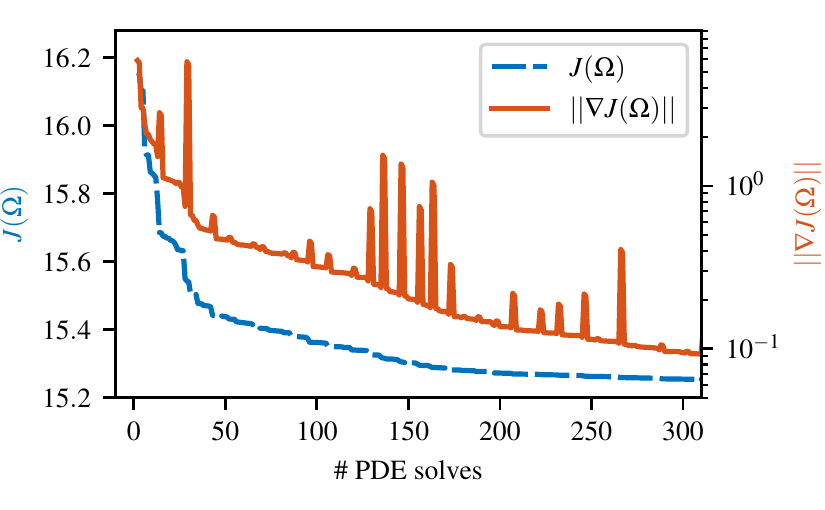}
    \end{center}
    \caption{\emph{Top left}: Optimal shape using cubic B-splines. \emph{Top right:} Optimal shape using as initial guess the domain obtained by flipping \cref{fig:elasticitycp} with respect to the horizontal axis.
    \emph{Bottom:} Convergence history for the left shape; the gradient is measured using the $H^1(D)$-norm.}
    \label{fig:elasticity-optimal}
\end{figure}
We use the same Lam\'e parameters as in \cite{AlPa06} ($E=15$, $\nu=0.35$) and obtain a qualitatively similar shape, which is shown in \cref{fig:elasticity-optimal}.

\section{Conclusion}
We have formulated shape optimization problems in terms of deformation diffeomorphisms.
This perspective simplifies the treatment of PDE-constrained shape optimization problems
because it couples naturally with isoparametric finite element discretization of the PDE-constraint.
In particular, it retains the asymptotic behavior of higher-order FE discretization, and it allows
the solution of PDE-constrained shape optimization problems to high accuracy, as confirmed by
detailed numerical experiments.
This shape optimization method can be implemented in standard finite element software and
used to tackle challenging shape optimization problems that stem from industrial applications.
The approach advocated is modular and can be combined with more advanced optimization algorithms,
such as that of Schulz et al.~\cite{ScSiWe16}; research in this vein will form the basis of
future work.


\end{document}